\documentclass[a4paper,11pt]{amsart}

\usepackage{amsmath,amsthm,amssymb}
\usepackage[mathscr]{eucal}
 \usepackage{cite}
\usepackage{upgreek}
\usepackage[bookmarks=false]{hyperref}
\usepackage{enumerate}
\usepackage{color}
\usepackage{graphicx}
\usepackage{amsmath,amsthm,amsfonts,amscd,amssymb,comment,eucal,latexsym,mathrsfs}
\usepackage{stmaryrd}
\usepackage[all]{xy}

\usepackage{epsfig}

\usepackage[all]{xy}
\xyoption{poly}
\usepackage{fancyhdr}
\usepackage{wrapfig}
\usepackage{epsfig}

\numberwithin{equation}{section}

\theoremstyle{plain}
\newtheorem{main}{Theorem}
\newtheorem{mcor}[main]{Corollary}

\newtheorem{theorem}{Theorem}[section]

\newtheorem{lemma}[theorem]{Lemma}
\newtheorem{proposition}[theorem]{Proposition}

\theoremstyle{definition}
\newtheorem{definition}[theorem]{Definition}

\newtheorem{notation}[theorem]{Notation}
\newtheorem{remark}[theorem]{Remark}

\theoremstyle{plain}
\newtheorem{thm}{Theorem}[section]
\newtheorem{prop}[thm]{Proposition}
\newtheorem{lem}[thm]{Lemma}

\theoremstyle{definition}

\theoremstyle{remark}

\newenvironment{Hoff_Commands}{

\newcommand{\id}{{\rm id}}
\newcommand{\<}{\langle}
\renewcommand{\>}{\rangle}
\newcommand{\textt}[1]{\quad\text{##1}\quad}
\newcommand{\y}{{\boldsymbol y}}
\renewcommand{\R}{\mc{R}}
\newcommand{\ii}{{\bf i}}
\newcommand{\CC}{\mathbb{C}}
\newcommand{\deltab}{{\boldsymbol\delta}}
\renewcommand{\E}{\text{E}}
\newcommand{\FF}{\mathbb{F}}
\renewcommand{\H}{\mc{H}}
\newcommand{\ten}{\otimes}
\renewcommand{\U}{\mc{U}}

\newcommand{\on}[1]{\stackrel{##1}{\curvearrowright}}
\newcommand{\ZZ}{\mathbb{Z}}
}{}

    \def\E{{\mathbb{E}}}   \def\H{{\mathbb{H}}}      \def\N{{\mathbb{N}}}    \def\R{{\mathbb{R}}}   \def\U{{\mathbb{U}}}     \def\Z{{\mathbb{Z}}}

  \def\cC{{\mathcal{C}}}   \def\cF{{\mathcal{F}}} \def\cG{{\mathcal{G}}}       \def\cN{{\mathcal{N}}}  \def\cP{{\mathcal{P}}}  \def\cR{{\mathcal{R}}} \def\cS{{\mathcal{S}}}       

                       \def\tX{{\tilde{X}}} \def\tY{{\tilde{Y}}} 

      \def\tPsi{{\widetilde{\Psi}}}       

 \def\tmu{{\widetilde{\mu}}}\def\tnu{{\widetilde{\nu}}}        \def\ttheta{{\widetilde{\theta}}}     

\newcommand\Aut{\operatorname{Aut}}

\newcommand\dom{\operatorname{dom}}

\newcommand\Prob{\operatorname{Prob}}

\def\cc{{\curvearrowright}}
\newcommand{\resto}{\upharpoonright}

\def\tcR{\tilde{\cR}}
\def\tcS{\tilde{\cS}}

\begin{document}

\title[Extensions of non-amenable equivalence relations]{von Neumann's problem and extensions of \\ non-amenable equivalence relations}

\author[Lewis Bowen]{Lewis Bowen}
\thanks{L.B. was partially supported by NSF grant DMS-1500389 and NSF CAREER Award DMS-0954606}
\address{Mathematics Department; University of Texas at Austin.}
\email{lpbowen@math.utexas.edu}

\author[Daniel Hoff]{Daniel Hoff}
\thanks{D.H. was partially supported by the NSF-GRFP under Grant No. DGE-1144086}
\address{Mathematics Department; University of California, San Diego, CA 90095-1555 (United States).}
\email{d1hoff@ucsd.edu}

\author[Adrian Ioana]{Adrian Ioana}
\thanks{A.I. was partially supported by NSF  Grant DMS 1161047, NSF CAREER Award DMS \#1253402, and a Sloan Foundation Fellowship}
\address{Mathematics Department; University of California, San Diego, CA 90095-1555 (United States).}
\email{aioana@ucsd.edu}

\begin{abstract} 

The goals of this paper are twofold. First, we generalize the result of Gaboriau and Lyons \cite{GL07} to the setting of von Neumann's problem for equivalence relations, proving that for any non-amenable ergodic probability measure preserving (pmp) equivalence relation $\mathcal{R}$, the Bernoulli extension over a non-atomic base space $(K, \kappa)$ contains the orbit equivalence relation of a free ergodic pmp action of $\mathbb{F}_2$. Moreover, we provide conditions which imply that this holds for any non-trivial probability space $K$.
Second, we use this result to prove that any non-amenable unimodular locally compact second countable group admits uncountably many free ergodic pmp actions which are pairwise not von Neumann equivalent (hence, pairwise not orbit equivalent).



 \end{abstract}
\maketitle

\section {Introduction and statement of main results}
\subsection*{Background}

The notion of amenability for groups was introduced by J. von Neumann in order to explain the Banach-Tarski paradox \cite{vN29}.  He showed that any countable group that contains the free group $\mathbb F_2$ on two generators is non-amenable. The question of whether any non-amenable group contains $\mathbb F_2$, became known as von Neumann's problem, and was eventually settled in the negative by A. Ol'shanskii \cite{Ol80}.

Remarkably, D. Gaboriau and R. Lyons  proved  that  von Neumann's problem has a positive solution in the context of measurable group theory \cite{GL07} (see also the survey \cite{Ho11}). More precisely, they showed that any countable non-amenable group $\Gamma$ admits $\mathbb F_2$ as a ``measurable subgroup": there exists a free ergodic probability measure preserving (pmp) action $\Gamma \curvearrowright (X, \mu)$ whose associated orbit equivalence relation $\mathcal{R}(\Gamma \curvearrowright X)$ contains the orbit equivalence relation of a free ergodic pmp action $\mathbb{F}_2 \curvearrowright (X, \mu)$. Moreover, the Bernoulli action of $\Gamma$ on $([0,1]^{\Gamma},\lambda^{\Gamma})$ has this property, where $\lambda$ denotes the Lebesgue measure on $[0,1]$. 

Our first goal is to establish the following strengthening of this result:

%

\begin{main}\label{main} Let $\mathcal R$ be an ergodic non-amenable countable pmp equivalence relation, and let $\mathcal{R}_K$ on $(X_K,\mu_{\kappa})$ denote its Bernoulli extension with base space $(K,\kappa)$.

If $(K,\kappa)$ is non-atomic, then there exists a free ergodic pmp action $\mathbb F_2\curvearrowright (X_K,\mu_{\kappa})$ such that $\mathcal R(\mathbb F_2\curvearrowright~X_K)~\le~\mathcal R_K,$ almost everywhere. 
Moreover, this conclusion holds for any non-trivial choice of $(K,\kappa)$, provided that $\mathcal R$ has an ergodic subequivalence relation of infinite index which is non-amenable or normal, or that $\mathcal R$ has an infinite fundamental group.
\end{main}

\begin{remark} 
Assume that $(K,\kappa)=([0,1],\lambda)$.
If $\mathcal R$ is the orbit equivalence relation of a free pmp action $\Gamma\curvearrowright X$, then  $\mathcal R_K$ is isomorphic to the orbit equivalence relation of the product action $\Gamma\curvearrowright X\times K^{\Gamma}$ (Proposition \ref{exam}).  In this case, Theorem \ref{main} is a consequence of \cite[Theorem 1]{GL07}. Indeed, since $\mathcal R(\Gamma\curvearrowright K^{\Gamma})$ contains the orbit equivalence relation of a free ergodic pmp action of $\mathbb F_2$ by \cite[Theorem 1]{GL07}, it follows that the same is true for $\mathcal R_K$. However, Theorem \ref{main} is new whenever $\mathcal R$ does not arise as the orbit equivalence relation of a free pmp action of a countable group (see \cite{Fu99} for examples of such $\mathcal R$).
Also, note that if $\mathcal R=\mathcal R(\Gamma\curvearrowright [0,1]^{\Gamma})$, then $\mathcal R_K$ is isomorphic to $\mathcal R$.
Theorem \ref{main} implies that $\mathcal R$ contains the orbits of a free ergodic pmp action of $\mathbb F_2$,  for any non-amenable $\Gamma$, and therefore recovers \cite[Theorem 1]{GL07}. 
\end{remark}

\begin{remark} 
At the end of \cite{GL07}, the authors posed the following analogue of von Neumann's problem for equivalence relations: does every ergodic non-amenable countable pmp equivalence relation $\mathcal R$ contain $\mathcal{R}(\mathbb{F}_2 \curvearrowright X)$ for some free ergodic pmp action of $\mathbb F_2$? The main result of \cite{GL07} shows that this is indeed the case if $\mathcal R$ arises from the Bernoulli action with base $([0,1],\lambda)$ of a non-amenable countable group. Theorem \ref{main} shows that, more generally, this holds for the Bernoulli extension with base $([0,1],\lambda)$ of any ergodic non-amenable countable pmp equivalence relation.
\end{remark}

We turn now to the second main result of this paper and to the history motivating it. 
In the early 1980s, D. Ornstein and B. Weiss \cite{OW80}, extending work of H. Dye \cite{Dy59}, showed that any two ergodic pmp actions of countable infinite amenable groups are orbit equivalent. Moreover, as a consequence of \cite{CFW81}, all free properly ergodic pmp actions of a unimodular amenable lcsc group $G$ are pairwise orbit equivalent.
On the other hand, over the next two decades, several families of non-amenable countable groups, including  property (T) groups \cite{Hj02} and non-abelian free groups \cite{GP03}, were shown to admit uncountably many actions which are pairwise not orbit equivalent. 

Unifying many of these results, it was shown in \cite{Io06} that any countable group $\Gamma$ containing a copy of $\mathbb{F}_2$ has uncountably many free ergodic actions which are pairwise not orbit equivalent. Thus nearly three decades after the solution to von Neumann's problem \cite{Ol80}, the relationship between general non-amenable groups and the prototypical example of $\mathbb{F}_2$ came again into focus. Gaboriau and Lyons' result in \cite{GL07} was followed shortly by \cite{Ep07}, in which I. Epstein combined \cite{GL07} with the methods of \cite{Io06} via a new co-induction construction for group actions, proving that any countable non-amenable group $\Gamma$ admits uncountably many non orbit equivalent actions, and settling the question in the case of countable groups.

%


Much less has been established in the case of unimodular lcsc groups $G$. It was shown in \cite[Example 5.2.13]{Zi84} (see also \cite[Corollary A. 10]{GG88}) that any connected semisimple Lie group $G$ with $\mathbb R$-rank$(G)\geq 2$, finite center, and no compact factors has uncountably many mutually non orbit equivalent free ergodic pmp actions. By combining \cite{Ep07} with an induction argument it follows that, more generally, any unimodular non-amenable lcsc group $G$ possessing a lattice has uncountably many non orbit equivalent free ergodic pmp actions. 
However, in spite of these advances, the situation for general non-amenable unimodular lcsc groups $G$ remained open. 

Making use of Theorem \ref{main}, we are able to settle this question: 

\begin{main}\label{lc}
Any unimodular non-amenable lcsc group $G$ admits uncountably many  free ergodic pmp actions which are pairwise not von Neumann equivalent (hence, pairwise not orbit equivalent). 
\end{main}

This will follow from Theorem \ref{T: ext} below on extensions of equivalence relations, a notion which will be key in the rest of the paper.

\begin{definition}
For countable pmp equivalence relations $\mathcal R$ on $(X, \mu)$ and $\tilde{\mathcal R}$ on $(\tilde X, \tilde \mu)$, we say that $\tilde{\mathcal R}$ is a {\it class-bijective extension} (in short, an {\it extension}) of $\mathcal R$ if there is a Borel map $p: \tilde X \to X$ satisfying
\begin{enumerate}
\item $\mu(E) = \tilde \mu(p^{-1}(E))$, for all Borel sets $E \subset X$,
\item $p|_{[x]_{\tilde{\mathcal R}}}$ is injective, for almost every $x \in \tilde X$, and 
\item $p([x]_{\tilde{\mathcal R}}) = [p(x)]_{\mathcal R}$, for almost every $x \in \tilde X$. \label{A: equal} 
\end{enumerate}
\end{definition}

\begin{remark}
A map $p:\tilde X\to X$ which satisfies conditions (1)-(3) in the above definition is called a {\it local OE}  (or {\it local isomorphism}) of $\tilde{\mathcal R}$, $\mathcal R$ in \cite[Definition 1.4.2]{Po05}.
\end{remark}

Theorem \ref{main} leads to the following characterization of non-amenability for ergodic equivalence relations in terms of actions of $\mathbb F_2$, which can be viewed as a weak version of von Neumann's problem for equivalence relations:

\begin{mcor}\label{cor}
An ergodic countable pmp equivalence relation $\mathcal R$ is non-amenable if and only if it admits an extension which contains almost every orbit of a free ergodic pmp action of $\mathbb F_2$. 
\end{mcor}



Combining this result with the co-induction construction of \cite{Ep07} and the methods of \cite{Io06}, we prove:

\begin{main}\label{T: ext}
Let $\mathcal R$ be a non-amenable ergodic countable pmp equivalence relation on a standard probability space.

Then $\mathcal R$ admits uncountably many ergodic extensions which are pairwise not stably von Neumann equivalent (hence, pairwise not stably isomorphic).
\end{main}

\begin{remark} Theorem \ref{T: ext} implies the following dichotomy: any ergodic countable pmp equivalence relation has either only one or uncountably many ergodic extensions, up to isomorphism. 
Indeed, if $\mathcal R$ is an amenable countable pmp equivalence relation, then $\mathcal R$ is hyperfinite by \cite{CFW81}.  As a consequence, any two ergodic extensions of $\mathcal R$ are hyperfinite, and thus isomorphic by \cite{Dy59}.  
\end{remark}

\begin{remark}
Let $\Gamma$ be a countable non-amenable group. If $\mathcal R$ is the orbit equivalence relation of some free pmp action of $\Gamma$, then any extension of $\mathcal R$ is the orbit equivalence relation of some other free pmp action of $\Gamma$. Theorem \ref{T: ext} implies that $\Gamma$ admits uncountably many actions which are pairwise not stably von Neumann equivalent,  thereby strengthening the results of \cite{Io06, Ep07}.

\end{remark}

Inspired by \cite{KPV13}, our approach to deducing Theorem \ref{lc} from Theorem \ref{T: ext} is based on the notion of cross section equivalence relations. Specifically, we rely on the following elementary observation: if  $\mathcal R$ is a cross section equivalence relation of some free ergodic pmp action of a unimodular lcsc group $G$, then any ergodic extension of ${\mathcal R}$ can be realized as a cross section equivalence relation of some other free ergodic pmp action of $G$ (see Proposition \ref{cross}).

This observation turns out to also be useful in a different context. 
Very recently, M. Gheysens and N. Monod introduced a measure-theoretic analogue of closed subgroup embeddings for locally compact groups, called {\it tychomorphism} \cite[Definition 14]{GM15}.  Using this notion, they formulated and proved a generalization of the Gaboriau-Lyons theorem for lcsc groups $G$: if $G$ is non-amenable, then there is a tychomorphism from  $\mathbb F_2$ to $G$ (see \cite[Theorem B]{GM15}). 
When combined with Theorem \ref{main}, the observation in Proposition \ref{cross} leads to a proof of this result which bypasses the usage of the structure theory of locally compact groups as in \cite{GM15} (see Subsection \ref{GM}).
 
We note that \cite{GM15} is representative of a lot of recent interest in measure-theoretic versions of von Neumann's problem spawned by the pioneering work \cite{GL07}. Thus, the main result of \cite{GL07} was strengthened in \cite{Ku13}, and its proof was simplified in \cite{Th13}. Very recently, von Neumann's problem was shown to have a positive solution for non-amenable equivalence relations that act on hyperbolic bundles \cite{Bo15}. 

\subsection*{Outline of the proof of Theorem \ref{main}} We end the introduction by outlining the proof of  the main assertion of Theorem \ref{main}.  This  relies on an extension of techniques from \cite{GL07}.
To fix notation, let $\mathcal R$ be an ergodic non-amenable countable pmp equivalence relation on a probability space $(X,\mu)$. Let $(K,\kappa)=([0,1],\lambda)$, and put $\tilde X:=X_K$ and $\tilde{\mathcal R}:=\mathcal R_{K}$. Our goal is to show that $\tilde{\mathcal R}$ contains the orbits of a free ergodic pmp action of $\mathbb F_2$.

To this end, denote by $u:[\mathcal R]\rightarrow\mathcal U(L^2(\mathcal R,m))$ the canonical representation of the full group $[\mathcal R]$.
Since $\mathcal R$ is ergodic and non-amenable, after  replacing $\mathcal R$ with a subequivalence relation, we may assume that $\mathcal R$ is generated by finitely many automorphisms $\theta_1,...,\theta_n\in [\mathcal R]$ such that the operator $T=\sum_{i=1}^n(u(\theta_i)+u(\theta_i^{-1}))$ satisfies $\|T\|<2n$. Moreover, after replacing the set $S=\{\theta_1,...,\theta_n\}$ with a power $S^k=\{\theta_{i_1}...\theta_{i_k}|1\leqslant i_1,...,i_k\leqslant n\}$, for large enough $k$, we may assume that $\|T\|\leqslant n$.

For $x\in X$, we denote by $\mathcal G_x=([x]_{\mathcal R},E_x)$ the graph  on $[x]_{\mathcal R}$ associated to the graphing $\{\theta_1,...,\theta_n\}$. Then we can identify the  $\tilde X$ with the set of pairs $(x,\omega)$, with $x\in X$ and $\omega\in [0,1]^{E_x}$, such that $\tilde{\mathcal R}$ is identified with the equivalence relation  given by: {$(x,\omega)\tilde{\mathcal R}(y,\xi)$ iff  $x\mathcal Ry$ and $\omega=\xi$.

For $p\in [0,1]$ and $x\in X$, we denote by $\pi_p:[0,1]^{E_x}\rightarrow \{0,1\}^{E_x}$ the map $\pi_p(\omega)=({\bf 1}_{[0,p]}(\omega_e))_e$, and
 view  $\pi_p(\omega)\in\{0,1\}^{E_x}$ as a subgraph of $\mathcal G_x$, for every $\omega=(\omega_e)_e\in\{0,1\}^{E_x}$.

In the first part of the proof, we use results from percolation theory, notably \cite{NS81} and \cite{BS96}, to show that if $p$ is in the interval $\big(\frac{1}{2n-\|T\|+1},\frac{1}{\|T\|}\big)$, then the graph $\pi_p(\omega)$ has infinitely many infinite clusters (i.e. connected components), for almost every $(x,\omega)\in\tilde X$. 

In the second part of the proof, we  consider the cluster equivalence relation $\tilde{\mathcal R}_{\text{cl}}$ on $\tilde X$ given by:  $(x,\omega){\tilde{\mathcal R}}_{\text{cl}}(y,\xi)$ iff $(x,\omega)\tilde{\mathcal R}(y,\xi)$ and} $x,y$ belong to the same cluster of $\pi_p(\omega)=\pi_p(\xi)$ \cite{Ga05}. By combining the first part of the proof with results from \cite{LS99,AL06} and \cite{Ga99} we conclude that the restriction ${\tilde{\mathcal R}}_{\text{cl}}$ to its infinite locus is ergodic and has normalized cost $>1$. 

Finally, since ${\tilde{\mathcal R}}_{\text{cl}}\subset\tilde{\mathcal R}$, a combination of results from \cite{Hj06}  and \cite{KM04,Pi05} implies that $\tilde{\mathcal R}$ contains the orbits of a free ergodic pmp action of $\mathbb F_2$.

\subsection*{Organization} 
Besides the introduction, this paper has seven other sections. In Section 2, we collect several facts about equivalence relation. In particular, we prove that we may assume $\|T\|\leqslant n$ and show that the isoperimetric constant of the graph $\mathcal G_x$ satisfies $\iota(\mathcal G_x)\geqslant 2n-\|T\|$. Section 3 contains various general results on Bernoulli extensions of equivalence relations.
 Sections 4 and 5 are devoted to the first and second part of the proof of the main assertion of Theorem \ref{main} described above. In Section 6, we complete the proof of Theorem \ref{main} and deduce Corollary \ref{cor}. Finally, in Sections 7 and 8, we present the proofs of Theorem \ref{T: ext} and Theorem \ref{lc}, respectively.

\subsection*{Acknowledgements} We are grateful to Nicolas Monod for helpful comments.


\section {Preliminaries}  

In this section we recall several general notions and results regarding equivalence relations.

\subsection{Equivalence relations} \label{equiv} Let $(X,\mu)$ be a probability space, always assumed to be standard. Following \cite{FM77}, an equivalence relation $\mathcal R$ on $X$ is called  {\it countable probability measure preserving} ({\it countable pmp}) if it has countable classes, $\mathcal R$ is a measurable subset of $X\times X$,  and any measurable automorphism $\theta:X\rightarrow X$ which satisfies $\theta(x)\in [x]_{\mathcal R}$, for almost every $x\in X$, preserves $\mu$.
Here, for $x\in X$, we denote by $[x]_{\mathcal R}$ its equivalence class. 

 The group of measurable automorphisms $\theta:X\rightarrow X$ satisfying $\theta(x)\in [x]_{\mathcal R}$, for almost every $x\in X$, is called the {\it full group} of $\mathcal R$ and denoted $[\mathcal R]$.
  We also denote by $[[\mathcal R]]$ the set of measurable isomorphisms $\theta:A\rightarrow B$ between  measurable subsets of $X$ which satisfy $\theta(x)\in [x]_{\mathcal R}$, for almost every  $x\in A$.

Here and after, we say that a pmp action $\Gamma\curvearrowright (X,\mu)$ is {\it essentially free} (in short, {\it free}) if the stabilizer $\Gamma_x=\{g\in\Gamma|g\cdot x=x\}$ is trivial, for almost every $x\in X$.
If $\Gamma\curvearrowright (X,\mu)$ is a pmp action of a countable group $\Gamma$, then its {\it orbit equivalence relation} $$\mathcal R(\Gamma\curvearrowright X):=\{(x,y)\in X\times X \;|\; \Gamma x=\Gamma y\}$$ is a countable pmp equivalence relation. Conversely, Feldman and Moore proved that any countable pmp equivalence relation arises this way \cite{FM77}. However, this action cannot always be taken to be free, a question that was settled by Furman \cite{Fu99}.

Let $\mathcal R$ be a countable pmp equivalence relation on $(X,\mu)$. We endow $\mathcal R$ with an infinite Borel measure $m$ given by $$m(A)=\int_{X}|\{y\in [x]_{\mathcal R}\;|\; (x,y)\in A\}|\;\text{d}\mu(x),\;\;\;\text{for every Borel subset}\;\;\;A\subset\mathcal R.$$
Then  $u:[\mathcal R]\rightarrow \mathcal U(L^2(\mathcal R,m))$ given by  the formula $(u({\theta})f)(x,y)=f(\theta^{-1}(x),y)$ defines a unitary representation.
Note that $L^{\infty}(\mathcal R)$ acts on $L^2(\mathcal R,m)$ by pointwise multiplication and the unitary $u(\theta)$ normalizes $L^{\infty}(\mathcal R)$, for every $\theta\in[\mathcal R]$.  We  also have an embedding $L^{\infty}(X)\subset L^{\infty}(\mathcal R)$ which associates to every $a\in L^{\infty}(X)$ the function $(x,y)\rightarrow a(x)$.

The von Neumann algebra of $\mathcal R$ is defined as the strong operator closure of the linear span of $\{a\hskip 0.03in u(\theta)|a\in L^{\infty}(X),\theta\in[\mathcal R]\}$ inside $\mathbb B(L^2(\mathcal R,m))$, and is denoted by $L(\mathcal R)$ \cite{FM77}. Recall that $L(\mathcal R)$ is a finite von Neumann algebra, with its canonical trace given by $$\tau(T)=\big\langle T({\bf 1}_{\Delta}),{\bf 1}_{\Delta}\big\rangle,\;\;\;\text{for every $T\in L(\mathcal R)$},$$
where ${\bf 1}_{\Delta}\in L^2(\mathcal R,m)$ denotes the characteristic function of $\Delta=\{(x,x)|x\in X\}$.

\begin{definition}
Let $\mathcal R$ and $\mathcal S$ be two (not necessarily countable) equivalence relation on probability spaces $(X,\mu)$ and $(Y,\nu)$. Then $\mathcal R$ and $\mathcal S$ are called {\it isomorphic} (resp. {\it stably isomorphic}) if there exist Borel subsets $X_0\subset X, Y_0\subset Y$ which are co-null (resp. complete sections for $\mathcal R$, $\mathcal S$),
and a measure preserving Borel isomorphism $\theta:X_0\rightarrow Y_0$ such that $x\mathcal Rx'$ iff $\theta(x)\mathcal S\theta(x')$, for all $x,x'\in X_0$.
Here, we endow $X_0\subset X$ with the probability measure $\mu(X_0)^{-1}(\mu|X_0)$. Also, we say that a Borel set $X_0\subset X$ is a complete section for $\mathcal R$ if $\{x\in X|\;[x]_{\mathcal R}\cap X_0\not=\emptyset\}$ is a co-null subset of $X$. 
Moreover, if $\mathcal R$ and $\mathcal S$  are countable pmp, then they are called {\it von Neumann equivalent} (resp. {\it stably von Neumann equivalent}) if their von Neumann algebras $L(\mathcal R)$ and $L(\mathcal S)$ are isomorphic (resp. $pL(\mathcal R)p\cong qL(\mathcal S)q$, for some non-zero projections $p\in L(\mathcal R)$, $q\in L(\mathcal S)$). 

Two pmp actions $\Gamma\curvearrowright (X,\mu)$ and $\Lambda\curvearrowright (Y,\nu)$ of two locally compact second countable ({\it lcsc}) groups  $\Gamma$ and $\Lambda$ are called  {\it orbit equivalent} (resp. {\it stably orbit equivalent}) if their orbit equivalence relations are isomorphic (resp. stably isomorphic).
Finally, the actions are called {\it von Neumann equivalent} if the associated crossed product von Neumann algebras (see, e.g.,  \cite[Chapter X]{Ta03} for the definition)  are isomorphic. 
\end{definition}

\subsection{Amenable equivalence relations} Following \cite[Definition 6]{CFW81}, a countable pmp equivalence relation $\mathcal R$ on $(X,\mu)$ is called {\it amenable} if there exists a state $\Phi:L^{\infty}(\mathcal R)\rightarrow\mathbb C$ such that $\Phi(u({\theta})fu({\theta})^*)=\Phi(f)$, for all $f\in L^{\infty}(\mathcal R)$, $\theta\in [\mathcal R]$, and $\Phi(a)=\int_{X}a\;\text{d}\mu$, for all $a\in L^{\infty}(X)$. 
 By \cite[Theorem 10]{CFW81} a countable pmp equivalence relation $\mathcal R$ is amenable if and only if it is {\it hyperfinite}. The latter means that we can write $\mathcal R=\cup_{n\geq 1}\mathcal R_n$, where $\mathcal R_n$ is a countable pmp equivalence relation on $X$ such that $[x]_{\mathcal R_n}$ is finite and $[x]_{\mathcal R_n}\subset [x]_{\mathcal R_{n+1}}$, for almost every $x\in X$ and all $n\geq 1$. 

Next, we record the well-known fact that an ergodic equivalence relation $\mathcal R$ is non-amenable if and only if the unitary representation $u:[\mathcal R]\rightarrow\mathcal U(L^2(\mathcal R,m))$ has {\it spectral gap}:

\begin{lemma}\label{spec} Let $\mathcal R$ be a non-amenable ergodic countable pmp equivalence relation.

Then we can find $n\geq 1$ and $\theta_1,...,\theta_n\in[\mathcal R]$ such that 
$\|\frac{1}{n}\sum_{i=1}^nu(\theta_i)\|<1.$ 
Moreover,  if $c>0$, then we can find $n\geq 1$ and $\theta_1,...,\theta_n\in[\mathcal R]$ such that 
$\|\frac{1}{n}\sum_{i=1}^nu(\theta_i)\|<c.$

\end{lemma} 
{\it Proof.} Assume by contradiction that $\|\frac{1}{n}\sum_{i=1}^nu(\theta_i)\|=1$, for all $\theta_1,...,\theta_n\in [\mathcal R]$. Then by arguing as in  the proof of \cite[Lemma 2.2]{Ha83} it follows  that there exists a state $\Phi:L^{\infty}(\mathcal R)\rightarrow\mathbb C$ such that  $\Phi(u({\theta})fu({\theta})^*)=\Phi(f)$, for all $f\in L^{\infty}(R)$ and every $\theta\in [\mathcal R]$. Since $\mathcal R$ is ergodic, any such state $\Phi$ also satisfies that $\Phi(g)=\int_{X}g\; \text{d}\mu$, for all $g\in L^{\infty}(X)$ (see e.g. the proof of \cite[Lemma 4.2]{HV12}). This implies that $\mathcal R$ is amenable, which is a contradiction.

For the moreover assertion, let $\theta_1,...,\theta_n\in [\mathcal R]$ such that $\delta:=\|\frac{1}{n}\sum_{i=1}^nu(\theta_i)\|<1$. Let $c>0$ and choose $m\geq 1$ such that $\delta^m<c$. Then we have that $$\left\|\frac{1}{n^m}\sum_{1\leq i_1,...,i_m\leq n}u(\theta_{i_1}...\theta_{i_m})\right\|=\left\|\left(\frac{1}{n}\sum_{i=1}^nu(\theta_i)\right)^m\right\|\leq\delta^m<c.$$

Thus, the elements $\theta_{i_1}...\theta_{i_m}\in [\mathcal R]$, for $1\leq i_1,...,i_m\leq n$, satisfy the second assertion.
\hfill$\square$

\subsection{Extensions and Expansions of Equivalence Relations}\begin{Hoff_Commands}
Let $\mathcal R$ and $\tilde{\mathcal R}$ be countable pmp equivalence relations on probability space $(X, \mu)$ and on $(\tilde X, \tilde \mu)$, respectively.

\begin{definition} We say that $\tilde{\mathcal R}$ is a {\it class-bijective extension} (in short, an {\it extension}) of $\mathcal R$ if there is a Borel  map $p: \tilde X \to X$ satisfying
\begin{enumerate}
\item $\mu(E) = \tilde \mu(p^{-1}(E))$, for all Borel $E \subset X$,
\item $p|_{[x]_{\tilde \R}}$ is injective, for almost every $x \in \tilde X$, and 
\item $p([x]_{\tilde \R}) = [p(x)]_\R$, for almost every $x \in \tilde X$. \label{A: equal} 
\end{enumerate}

We say that $\tilde \R$ is an {\it expansion} of $\R$ if condition (\ref{A: equal}) is weakened to

\begin{enumerate}
\item[(3')] $p([x]_{\tilde \R}) \supset [p(x)]_\R$ for almost every $x \in \tilde X$.
\end{enumerate}
\end{definition}

\begin{notation} Below we use the notation $\tilde{\mathcal R}\rightarrow\mathcal R$ to mean that $\tilde{\mathcal R}$ is an extension of $\mathcal R$.
\end{notation}

\begin{remark}\label{R: lift}
Assume that $\tilde \R$ is an extension of $\mathcal R$ and let $\mathcal S\leq\mathcal R$ be a subequivalence relation.
Then $\tilde{\mathcal S}:=\{(x,y)\in\tilde \R | (p(x),p(y))\in\mathcal S\}$  is an extension of $\mathcal S$, which we call the {\it lift of $\mathcal S$ to $\tilde{\mathcal R}$}.
\end{remark}

\begin{remark}
Assume that $\tilde \R$ is an expansion of $\mathcal R$. Then $\tilde\R$ contains an extension $\tilde \R_0 \leq \tilde \R$ of $\R$ defined by $\tilde \R_0 = \{(x, y) \in \tilde \R | (p(x), p(y)) \in \R\}.$ Note, however, that containing an extension of $\R$ is not equivalent to being an expansion of $\R$. 
\end{remark}

Suppose that $\tilde \R$ is an expansion of $\R$ and let $p:\tilde X\rightarrow X$ as in the above definition. If $\theta \in [\R]$, then for almost every $x \in \tilde X$, the set $p^{-1}(\theta(p(x)))\cap [x]_{\tilde \R}$ contains exactly one point $x' \in \tilde X$. We may therefore define $\tilde\theta \in [\tilde \R]$ by $\tilde\theta(x) = x'$. 
Note that $\theta\circ p = p \circ \tilde\theta$, for all $\theta \in [\R]$. One can check that $\theta\mapsto \tilde\theta$ is a homomorphism from $[\R]$ into $[\tilde \R]$. 
For $a \in L^\infty(X)$, we let $\tilde a = a \circ p \in L^\infty(\tilde X)$. 

The next result is due to S. Popa \cite[Proposition 1.4.3]{Po05}. For completeness, we include a proof.

\begin{lemma}\emph{\cite{Po05}}\label{L: exp}
There is a trace preserving $*$-homomorphism $\pi: L(\R) \rightarrow L(\tilde \R)$ satisfying
 \begin{enumerate}
\item $\pi(a) = \tilde a$, for every $a \in L^\infty(X)$,
\item $\pi(u(\theta)) = u({\tilde\theta})$, for every $\theta\in [\R]$, and
\item $\pi(L^\infty(X))' \cap L(\tilde \R) = L^\infty(\tilde X).$ \label{LC: commutant}
\end{enumerate}
Moreover, if $\tilde \R$ is an extension of $\R$, then
the linear span of $\{bu({\tilde\theta})| b \in L^\infty(\tilde X), \theta \in [\R]\}$ is dense in $L(\tilde{\mathcal R})$, in the strong operator topology.

\end{lemma}
\begin{proof}
We denote by $\tau$ and $\langle.,.\rangle$  the canonical trace and inner product on $L(\mathcal R)$ (resp. $L(\mathcal{\tilde R})$) and by $\E_{L^{\infty}(X)}$ (resp. $\E_{L^{\infty}(\tilde X)}$) the conditional expectations onto $L^{\infty}(X)$ (resp. $L^{\infty}(\tilde X)$).
Note first that the map $\pi: L^\infty(X) \to L^\infty(\tilde X)$ given by $\pi(a) = \tilde a$ defines a trace preserving $*$-homomorphism. 
Moreover, if $\theta \in [\R]$, then  
\begin{align*}
\pi(\E_{L^{\infty}(X)}(u({\theta}))) = {\bf 1}_{\{x \in X|\theta(x) = x\}} \circ p 
= {\bf 1}_{\{x \in \tilde X| \theta(p(x)) = p(x)\}} 
= {\bf 1}_{\{x \in \tilde X: \tilde\theta(x) = x\}} 
= \E_{L^{\infty}(\tilde X)}(u({\tilde\theta})).
\end{align*}
Let $D \subset L(\R)$ be the $*$-subalgebra consisting of finite sums of the form $\sum_{\theta} a_\theta u({\theta})$. 
Then $D\subset L(\R)$ is dense in the strong operator topology, and for every $a, b \in L^\infty(X)$, $\theta,\rho \in [\R]$, we have
\begin{align*}
\<\pi(a) u({\tilde\theta}), \pi(b) u({\tilde\rho})\> 
&= \tau(\pi(b^*a) u({\tilde\theta \tilde\rho^{-1}})) 
= \tau( \pi(b^*a) \E_{L^{\infty}(\tilde X)}(u({\widetilde{\theta\rho^{-1}}})))
\\ &= \tau(\pi(b^*a) \pi(\E_{L^{\infty}(X)}(u({\theta\rho^{-1}})))) 
= \tau(b^*a\E_{L^{\infty}(X)}(u({\theta\rho^{-1}})))
= \<au({\theta}), b u({\rho})\>. 
\end{align*}
Therefore, the map $\pi: D \to L(\tilde \R)$ defined by $\sum_{\theta} a_\theta u({\theta}) \mapsto \sum_{\theta}\tilde a_\theta u({\tilde\theta})$ is well-defined and trace-preserving. Moreover,  since 
$\pi(u(\theta)a{u(\theta})^*) = \pi(a \circ \theta^{-1}) = \pi(a) \circ \tilde{\theta}^{-1}
= u({\tilde\theta})au({\tilde\theta})^*$
and the maps $a \mapsto \tilde a$ and $u({\theta}) \mapsto u({\tilde\theta})$ are $*$-homomorphisms, $\pi$ is a $*$-homomorphism.
Since $\pi$ is trace-preserving, it extends to a trace-preserving $*$-homomorphism $\pi: L(\R) \to L(\tilde \R)$ satisfying (1) and (2). 

To prove (\ref{LC: commutant}), let $y \in \pi(L^\infty(X))' \cap L(\tilde \R)$. 
Fix $\theta \in [\tilde \R]$ and set $b_\theta= \E_{L^{\infty}(\tilde X)}(yu(\theta)^*)$. 
Then for any $a \in\pi(L^{\infty}(X))$ we have
\begin{align*}
b_{\theta}a = ab_{\theta} = \E_{L^{\infty}(\tilde X)}(ayu({\theta})^*) = \E_{L^{\infty}(\tilde X)}(yau({\theta})^*) = b_{\theta}(u({\theta})au({\theta})^*).
\end{align*}
Thus, for almost every $x \in \text{supp}(b_\theta) \subset \tilde X$, we have $p(x) = p(\theta^{-1}(x))$. Since   $p|_{[x]_{\tilde \R}}$ is injective, we derive that $x = \theta^{-1}(x)$, for almost every $x \in \text{supp}(b_\theta)$. Hence, for any $b \in L^{\infty}(\tilde X)$ and almost every $x \in \tilde X$, we have $b_{\theta}(x)[u(\theta) bu({\theta})^*](x) = b_{\theta}(x)b({\theta}^{-1}x) = b_{\theta}(x)b(x)$. Therefore, we get that $b_{\theta}u({\theta}) \in L^{\infty}(\tilde X)' \cap L(\tilde{\mathcal R})= L^{\infty}(\tilde X)$. Since this holds for any $\theta\in [\tilde{\mathcal{R}}]$, we conclude that $y \in L^{\infty}(\tilde X)$.

Finally, assume that $\tilde \R$ is an extension of $\R$, and let $\theta \in [\tilde \R]$. Then for almost every $x \in \tilde X$ we have $(p(x), p(\theta(x))) \in \R$.  Hence
$\tilde X = \bigcup_{\rho \in [\R]} \{x \in \tilde X| p(\theta(x)) = \rho(p(x))\}
= \bigcup_{\rho \in  [\R]} \{x \in \tilde X|\theta(x) = \tilde\rho(x)\}.$
Thus, we can write $u_\theta = \sum_{n = 1}^\infty z_nu({\tilde\rho_n})$, for some $\{\rho_n\} \subset [\R]$ and projections $\{z_n\} \subset L^{\infty}(\tilde X)$ with $\sum_{n = 1}^\infty z_n = 1$. 
In particular, this gives the moreover conclusion.
\end{proof}
\end{Hoff_Commands}

\subsection{Graphed equivalence relations and isoperimetric constants}\label{isop}
Let $\mathcal R$ be a countable pmp equivalence relation on a probability space $(X,\mu)$. 
  A {\it graphing} of $\mathcal R$ is an at most countable family $\{\theta_i\}_{i\geq 1}\subset [[\mathcal R]]$. A graphing $\{\theta_i\}_{i\geq 1}$ is {\it generating}  if $\mathcal R$ is the smallest equivalence relation which contains the graph of $\theta_i$ for all $i\geq 1$.

Any graphing $\{\theta_i:A_i\rightarrow B_i\}_{i\geq 1}$  gives rise to a graph structure on $\mathcal R$  (see \cite{Ga99}). 
More precisely, for $x\in X$, we define an unoriented (multi-)graph $\mathcal G_{x}=([x]_{\mathcal R},E_x)$ whose vertex set is the equivalence class $[x]_{\mathcal R}$ and whose edge set $E_x$ consists of the pairs $(y,\theta_i(y))$, for every $i\geq 1$ and $y\in [x]_{\mathcal R}\cap A_i$.  Note  that we allow multiple edges between two given vertices.  Therefore, if the graphing is finite and given by $\{\theta_i\}_{i=1}^n$, with $\theta_1,...,\theta_n\in [\mathcal R]$, then $\mathcal G_x$ is a $2n$-regular graph.

Let $\mathcal G=(V,E)$ be an unoriented infinite (multi-)graph with vertex set $V$ and edge set $E$. Given a non-empty finite set $F\subset V$, let $\partial_{E}F$ be the set of edges which have exactly one endpoint in $F$. The {\it edge-isoperimetric} constant of $\mathcal G$ is defined as $$\iota(\mathcal G)=\inf\left\{\frac{|\partial_{E}F|}{|F|}\;\Big |\;\; \emptyset\not=F\subset V\;\text{finite subset}\right\}.$$ 

If $\mathcal R$ is an amenable countable pmp equivalence relation, then $\iota(\mathcal G_x)=0$, for almost every $x\in X$, for any finite graphing $\{\theta_i\}_{i=1}^n$ (see \cite[Theorem 2]{Ka97}). On the other hand, the converse is false. More precisely, \cite[Section 3]{Ka97} provides an example of a non-amenable equivalence relation $\mathcal R$ which admits a finite generating graphing $\{\theta_i\}_{i=1}^n$ such that  $\iota(\mathcal G_x)=0$, for almost every $x\in X$. 

Nevertheless, the combination of Lemma \ref{spec} and Lemma \ref{iso} below shows that if $\mathcal R$ is non-amenable and ergodic, then we can find a graphing $\{\theta_i\}_{i=1}^n$ with $\theta_1,...,\theta_n\in [\mathcal R]$ such that  the associated graphs satisfy $\iota(\mathcal G_x)>0$, for almost every $x\in X$. 

\begin{lemma}\label{iso}
Let $\mathcal R$ be a countable pmp equivalence relation on a probability space $(X,\mu)$ and $\theta_1,...,\theta_n\in [\mathcal R]$. For every $x\in X$, consider the unoriented graph $\mathcal G_{x}=([x]_{\mathcal R},E_x)$ defined as above.

Then $\iota(\mathcal G_x)\geq 2n-\|\sum_{i=1}^n(u(\theta_i)+u(\theta_i^{-1}))\|$, for almost every $x\in X$.
\end{lemma}

{\it Proof.}  Denote $\delta=2n-\|\sum_{i=1}^n(u(\theta_i)+u(\theta_i^{-1}))\|$. Let $S$ be the set of $y\in X$ such that $\iota(\mathcal G_y)<\delta$.  Assume by contradiction that $\mu(S)>0$. For all $y\in S$ we can find a finite set $A_y\subset [y]_{\mathcal R}$ satisfying $|\partial_{E_y}(A_y)|<\delta |A_y|$ in such a way that the set $A:=\{(x,y)\in\mathcal R|y\in S,x\in A_y\}$ is Borel. Moreover, after replacing $S$ with a non-null Borel subset, we may assume that $\sup_{y\in S}|A_y|<\infty$.

If we view ${\bf 1}_A\in L^2(\mathcal R,m)$,  then $\langle u(\theta)({\bf 1}_A),{\bf 1}_A\rangle=\int_{S}|\{x\in A_y|\theta^{-1}(x)\in A_y\}|\;\text{d}\mu(y)$, for all $\theta\in[\mathcal R]$. 
By using this identity we derive that \begin{align*}&\left\langle \left(\sum_{i=1}^n(u(\theta_i)+u(\theta_i^{-1}))\right)({\bf 1}_A),{\bf 1}_A\right\rangle \\&=\int_{S}\sum_{x\in A_y}\big(|\{1\leq i\leq n|\theta_i^{-1}(x)\in A_y\}|+|\{1\leq i\leq n|\theta_i(x)\in A_y\}|\big)\;\text{d}\mu(y)\\&=\int_{S}\big(2n|A_y|-|\partial_{E_y}(A_y)|\big)\;\text{d}\mu(y)>(2n-\delta)\int_{S}|A_y|\;\text{d}\mu(y)\\&=(2n-\delta)\; m(A)=(2n-\delta)\langle {\bf 1}_A,{\bf 1}_A\rangle.\end{align*}

This contradicts the fact that $\|\sum_{i=1}^n(u(\theta_i)+u(\theta_i^{-1}))\|=2n-\delta.$\hfill$\square$

\subsection{Cost of equivalence relations} 
Let $\mathcal R$ be a countable pmp equivalence relation on a probability space $(X,\mu)$. The cost of a graphing $\{\theta_i:A_i\rightarrow B_i\}_{i\geq 1}$  is the sum of the measures of the domains: $\sum_{i\geq 1}\mu(A_i)$.  The {\it cost} of $\mathcal R$ is defined as  the infimum of the cost of all generating graphings  of $\mathcal R$ \cite[Defintion I.5]{Ga99}.

Let $A\subset X$ be a Borel set of positive measure and denote by $\mathcal R\resto A:=\mathcal R\cap (A\times A)$ the {\it restriction of $\mathcal R$ to $A$}. Then the {\it normalized cost} of $\mathcal R\resto A$ is defined as the cost of $\mathcal R\resto A$ with respect to the  probability measure on $A$ given by $\mu_A(B)=\mu(B)/\mu(A)$, for any Borel set $B\subset A$. 

In the proof of our main result we will use the following theorem. 

\begin{theorem}\label{hj}
Assume that $\mathcal R$ is ergodic and has cost in $(1,\infty)$.
Then there exists a free ergodic pmp action $\mathbb F_2\curvearrowright (X,\mu)$ such that $\mathcal R(\mathbb F_2\curvearrowright X)\subset\mathcal R$, almost everywhere.
\end{theorem}

This theorem is the combination of Propositions 13 and 14 from \cite{GL07}. Its proof relies on  a theorem due to G. Hjorth \cite{Hj06}  and on a result from \cite{KM04,Pi05} (see \cite{GL07} for details).


\section{Bernoulli extensions of equivalence relations}

In this section, we first recall the construction of Bernoulli extensions and prove that Bernoulli extensions preserve ergodicity. 
We then study isomorphisms of Bernoulli extensions and their behavior with respect to restrictions to subequivalence relations and compressions.

\subsection{Bernoulli extensions and ergodicity} 
Let $\mathcal R$ be a countable pmp equivalence relation on a probability space $(X,\mu)$. Let $(K,\kappa)$ be a probability space.

We denote by $X_K$ the set of pairs $(x,\omega)$ with $x\in X$ and $\omega\in K^{[x]_{\mathcal R}}$. We endow $X_K$ with the smallest $\sigma$-algebra of sets which makes the maps $(x,\omega)\mapsto x$ and $(x,\omega)\mapsto\omega(\theta(x))$ measurable, for every $\theta\in [\mathcal R]$. We also endow $X_K$ with the probability measure $\mu_{\kappa}$ given by
$$d\mu_\kappa(x,\omega) = d\kappa^{[x]_{\mathcal R}}(\omega) ~d\mu(x).$$
Lastly, we denote by $\cR_K$ the equivalence relation on $X_K$ given by $(x,\omega) \cR_K (y,\xi)$ iff $x\cR y$ and $\omega =\xi$, and call it the {\it Bernoulli extension of $\mathcal R$ with base space $(K,\kappa)$} (see \cite[Section 11]{Bo14}).

 \begin{lemma}\label{ergodic} If $\mathcal R$ is ergodic, then $\mathcal R_K$ is ergodic. 
 \end{lemma}
 
 {\it Proof.} Assume that $\mathcal R$ is ergodic. Then we can find $\theta\in [\mathcal R]$ which acts ergodically on $(X,\mu)$ (see \cite[Theorem 3.5]{Ke10}). We define $\tilde\theta\in [\mathcal R_K]$ by letting $\tilde{\theta}(x,\omega)=(\theta(x),\omega)$, for $(x,\omega)\in X_K$. Then, in order to conclude that $\mathcal R_K$ is ergodic, it suffices to prove that $\tilde\theta$ acts ergodically on $X_K$. 
 
 Let $\mathcal S\leq\mathcal R$  be the subequivalence relation generated by $\theta$. Since $\theta$ and hence $\mathcal S$ is ergodic, we can find $\{\theta_i\}_{i=1}^N\in [\mathcal R]$ such that for almost every $x\in X$ we have $\theta_i([x]_{\mathcal S})\cap\theta_j([x]_{\mathcal S}) = \emptyset$, for all $i\not=j$, and $[x]_{\mathcal R}=\cup_{i=1}^N\theta_i([x]_{\mathcal S})$ (see \cite[Lemma 1.1]{Io09}). Here, $N\in\mathbb N\cup\{\infty\}$ is the index of $\mathcal S$ in $\mathcal R$. 
 
 Now, we define $\sigma:X\times K^{\{1,...,N\}\times\mathbb Z}\rightarrow X_K$ by letting $\sigma(x,(k_{i,j})_{i\in\{1,...,N\},j\in\mathbb Z})=(x,\omega)$, where $\omega \in K^{[x]_{\mathcal R}}$ is given by $\omega(\theta_i\theta^j(x))=k_{i,j}$, for all $i\in\{1,...,N\}$ and $j\in\mathbb Z$. Further, we endow $X\times K^{\{1,...,N\}\times\mathbb Z}$ with the probability measure $\mu\times\kappa^{\{1,...,N\}\times\mathbb Z}$. Then it is clear that $\sigma$ is an isomorphism of probability spaces and that $$(\sigma^{-1}\circ\tilde\theta\circ\sigma)(x,k)=(\theta(x),(k_{i,j+1})),\;\;\;\text{for all}\;\; x\in X,\; k=(k_{i,j})_{i\in\{1,...,N\},j\in\mathbb Z}\in K^{\{1,...,N\}\times\mathbb Z}.$$ 
 
 Thus, $\tilde\theta$ is conjugate to the product $\theta\times\tau$ between $\theta$ and the Bernoulli shift  $\tau$ of $\mathbb Z$ on $(K^{\{1,...,N\}})^{\mathbb Z}$. Since $\theta$ is ergodic and $\tau$ is weakly mixing, we conclude that $\tilde\theta$ is ergodic. \hfill$\square$

Let us also note that if $\mathcal R$ is the orbit equivalence relation of some free action, then the Bernoulli extensions of $\mathcal R$ can be described explicitly.

 \begin{proposition}\label{exam}
 Assume that $\mathcal R=\mathcal R(\Gamma\curvearrowright X)$, for some essentially free pmp action $\Gamma\curvearrowright (X,\mu)$.

Then $\mathcal R_K$ is isomorphic to $\mathcal R(\Gamma\curvearrowright X\times K^{\Gamma})$.
 \end{proposition}
{\it Proof.} Let $\theta:X_K\rightarrow X\times K^{\Gamma}$ be given by $\theta(x,\omega)=(x,\eta)$, where $\eta(g)=\omega(g^{-1}x)$, for every $g\in\Gamma$. It is immediate to see that $\theta$ is an isomorphism of probability spaces which implements an isomorphism between $\mathcal R_K$ and $\mathcal R(\Gamma\curvearrowright X\times K^{\Gamma})$.
\hfill$\square$

\subsection{Isomorphisms of Bernoulli extensions}\label{sec:isom}
Let $\mathcal R$ be a countable pmp equivalence relation on a probability space $(X,\mu)$.
Next, we study the isomorphism problem for Bernoulli extensions. For this we need the following definition:
\begin{definition}[Isomorphism of extensions]
Let $\tcR$, $\tcS$ be countable pmp equivalence relations on probability spaces $ (\tX,\tmu)$ and $(\tY,\tnu)$. Suppose that  $\pi:\tX \to X$, $\phi:\tY \to X$ are Borel maps which give  extensions of $\tcR$ and $\tcS$ over $\cR$, respectively. We say that the extensions $\tcR \to \cR$ and $\tcS \to \cR$  are {\it isomorphic} if there is an isomorphism $\psi:\tX \to \tY$ of $\tcR$ with $\tcS$ such that $\pi = \phi\circ \psi$. 
\end{definition}
Let $(K,\kappa)$ be a standard probability space. If $(K,\kappa)$ is purely atomic then we define its {\bf Shannon entropy} by
$$H(K,\kappa):=\sum_{k\in K} - \mu(\{k\}) \log (\mu(\{k\})).$$
By convention $0\cdot \log(0)=0$. Otherwise, we set $H(K,\kappa):=+\infty$. 

\begin{thm}
\label{thm:isomorphism} 
 Let $(K,\kappa), (L,\lambda)$ be probability spaces with the same Shannon entropy. 
 
 Then the corresponding Bernoulli extensions  of $\cR$ are isomorphic.  
\end{thm}

\begin{proof}
By Ornstein's Isomorphism Theorem \cite{Or70a, Or70b} the Bernoulli shifts $\Z \cc (K,\kappa)^\Z$ and $\Z \cc (L,\lambda)^\Z$ are isomorphic. Let $\Phi:K^\Z \to L^\Z$ be such an isomorphism. 

Fix an aperiodic element  $\theta \in [\cR]$ (see \cite[Theorem 3.5]{Ke10} for the proof of existence). Let $x\in X$ and $\omega\in K^{[x]_{\cR}}$. For $y \in [x]_\cR$, let $\omega^y \in K^\Z$ be the map given by
$$\omega^y(n) = \omega(\theta^n y).$$
Also, define
$\omega':[x]_\cR \to L$ by
$$\omega'(y) = \Phi(\omega^y)_0.$$
That is, $\omega'$ is the time 0 coordinate of $\Phi(\omega^y)$. Finally, define $\Psi: X_K \to X_L$ by 
$$\Psi(x,\omega)= (x,\omega').$$ 
The proof that $\Psi$ gives the desired isomorphism is similar to the proof of \cite[Lemma 3.1]{Bo12}. For the reader's convenience, the proof is sketched below.

To prove invertibility, given $\omega \in L^{[x]_\cR}$, define $\omega'':[x]_\cR \to K$ by
$$\omega''(y) = \Phi^{-1}(\omega^y)_0.$$
We claim that $(\omega')''=\omega$. Indeed:
$$(\omega')''(y) = \Phi^{-1}( (\omega')^y)_0 = \Phi^{-1}(n \mapsto \omega'(\theta^n y))_0 = \Phi^{-1}(n \mapsto \Phi(\omega^{\theta^n y})_0)_0.$$
Because
$$(\omega^{\theta^n y})_m = \omega(\theta^{m+n}y) = \omega^y_{n+m}$$
and $\Phi$ is shift-equivariant,
$$ \Phi(\omega^{\theta^n y})_0 = \Phi(\omega^{y})_n.$$
Plug this back into the formula for $(\omega')''(y)$ to obtain
$$(\omega')''(y) = \Phi^{-1}(n \mapsto \Phi(\omega^{y})_n)_0 = \Phi^{-1}(\Phi(\omega^{y}))_0=\omega^y_0 = \omega(y).$$
Now define $\tPsi:X_L \to X_K$ by $\tPsi(x,\omega)=(x,\omega'')$. By the previous computation,
$$\tPsi(\Psi(x,\omega)) = (x,(\omega')'') = (x,\omega).$$
Similarly, $\Psi \tPsi$ is also the identity so $\tPsi$ is the inverse of $\Psi$.

Fix $x\in X$. It suffices to show that $\Psi$ maps the fiber measure $\kappa^{[x]_\cR}$ to the fiber measure $\lambda^{[x]_\cR}$. So let $\omega:[x]_\cR \to K$ be random with law $\kappa^{[x]_\cR}$. The restrictions of $\omega$ to the orbits of $\theta$ are jointly independent. Because $\omega'(y)$ depends only on the $\theta$-orbit of $y \in [x]_\cR$, the restrictions of $\omega'$ to the orbits of $\theta$ are also jointly independent. The law of any $\theta$-orbit is $\kappa^\Z$. Since $\Phi_*\kappa^\Z=\lambda^\Z$, it follows that $\Psi $ maps the fiber measure $\kappa^{[x]_\cR}$ to $\lambda^{[x]_\cR}$ as required.

\end{proof}

\begin{definition}
Theorem \ref{thm:isomorphism} allows us to define the {\it Bernoulli extension of $\cR$ with base entropy $t \in (0,\infty]$} to be any Bernoulli extension of $\cR$ with base space $(K,\kappa)$ satisfying $H(K,\kappa)=t$. 
\end{definition}

\subsection{Bernoulli extensions restricted to subequivalence relations} 

\begin{thm}
\label{thm:subeq}
Let $\cR$ be a countable ergodic pmp equivalence relation on a probability space $(X,\mu)$. Let $(K,\kappa)$ be a probability space and $\cR_K \to \cR$ the corresponding Bernoulli extension of $\cR$. Let $\cS\le \cR$ be an ergodic subequivalence relation and $\tcS \le \cR_K$ the lift of $\cS$ to $\cR_K$. 

Then the extension $\tcS \to \cS$ is isomorphic to the Bernoulli extension of $\cS$ with base space entropy equal to $H(K,\kappa)[\cR:\cS]$. 
\end{thm}

\begin{proof}
Let $N=[\mathcal R:\mathcal S]$.
Since $\mathcal S$ is ergodic, we can find  $\{\theta_i\}_{i=1}^N\in [\mathcal R]$ such that for almost every $x\in X$ we have $\theta_i([x]_{\mathcal S})\cap\theta_j([x]_{\mathcal S}) = \emptyset$, for all $i\neq j$, and $[x]_{\mathcal R}=\cup_{i=1}^N\theta_i([x]_{\mathcal S})$ (see \cite[Lemma 1.1]{Io09}). 

Let $(L,\lambda)=(K,\kappa)^{N}$. We denote by $(Y_L,\nu_{\lambda})$ the underlying space of the Bernoulli extension $\mathcal S_L$ the Bernoulli extension of $\mathcal S$ with base space $(L,\lambda)$. Specifically, $Y_L=\{(x,\omega')|x\in X,\omega'\in L^{[x]_{\mathcal S}}\}$.

Define the isomorphism $\Phi:X_K \to Y_L$ by letting $\Phi(x,\omega) = (x, \omega')$, where $\omega': [x]_\cS \to L$ is defined by $\omega'(y)_n = \omega(\theta_n(y))$. We will prove that  $\Phi$ is an isomorphism between the extensions $\tcS \to \cS$ and $\cS_L \to \cS$. Since $H(L,\lambda)=N\;H(K,\kappa)=[\mathcal R:\mathcal S]\;H(K,\kappa)$, the conclusion follows.

For $(x,\omega') \in Y_L$, define $\omega'':[x]_\cR \to K$ by $\omega''(\theta^n(y)) =\omega'(y)_n$ for $y\in [x]_\cS$. This is well-defined because $\{\theta_n(y):~y\in [x]_\cS\}$ partitions $[x]_\cR$. Define $\Psi:Y_L \to X_K$ by $\Psi(x,\omega')=\omega''$. Then $\Psi$ is the inverse of $\Phi$, so $\Phi$ is invertible.

Fix $x\in X$ and let $\omega$ be a random variable with law $\kappa^{[x]_\cR}$. Since $\{\omega(y)\}_{y \in [x]_\cR}$ are i.i.d. random variables and $\{\theta_n(y):~y\in [x]_\cS\}$ partitions $[x]_\cR$, it follows that $\{\omega'(y)\}_{y \in [x]_\cS}$ are also i.i.d. random variables. So $\Phi$ maps the fiber measure $\kappa^{[x]_\cR}$ to the fiber measure $\lambda^{[x]_\cS}$ and therefore, it maps $\mu_\kappa$ to $\mu_\lambda$.

\end{proof}

\subsection{Compressions of Bernoulli extensions}

\begin{thm}
\label{thm:compression}
Let $\cR$ be a countable ergodic pmp equivalence relation on a probability space $(X,\mu)$. Let $(K,\kappa)$ be a probability space and $\cR_K \to \cR$ the corresponding Bernoulli extension of $\cR$. Let $Y \subset X$ be a non-null Borel set and let $\tY \subset X_K$ be the corresponding lift.

 Then the extension $\cR_K \resto \tY \to \cR \resto Y$ is isomorphic to the Bernoulli extension of $\cR \resto Y$ with base space entropy equal to $H(K,\kappa)/\mu(Y)$. 
\end{thm}

To prove this we first need to study the classification of inhomogeneous Bernoulli shifts.

Let $T \in \Aut(X,\mu)$ be an ergodic automorphism of a probability space. Let $\Omega$ be a complete metric space and $\Prob(\Omega)$ the set of probability measures on $\Omega$ endowed with the weak* topology. Suppose that $\phi:X \to \Prob(\Omega)$ is a Borel map. For $x\in X$, let $\kappa_x$ be the probability measure on $\Omega^\Z$ obtained as the direct product of the measures $\phi(T^nx)$ ($n\in \Z$).



Define the measure $\tmu$ on $X \times \Omega^\Z$ by 
$$d\tmu(x,\omega) = d\kappa_x(\omega)~d\mu(x).$$
We let $\sigma$ denote the shift map from $\Omega^\Z$ to itself given by $\sigma(\omega)_n = \omega_{n+1}$. Then $\tmu$ is $T\times \sigma$-invariant. The automorphism $T\times \sigma$ is called the {\it inhomogeneous Bernoulli shift over $T$ with data $\phi$}.

\begin{lem}\label{lem:inhomog}
The inhomogeneous Bernoulli shift defined above is measurably conjugate to a direct product $T \times U$, where $U$ is a Bernoulli shift with  $h(U) = \int H(\phi(x))~d\mu(x).$ Moreover, the conjugacy can be chosen to be the identity on the $X$ coordinate. 
\end{lem}

\begin{proof}
This is a straightforward consequence of Thouvenot's Relative Isomorphism Theorem \cite{Th75}. We provide some details here, guided by the formulation of Thouvenot's Theorem presented in \cite{Ki84}. 

Let $(B,\rho)$ be a complete separable metric space. If $\nu_1,\nu_2$ are probability measures on $B^m$ (for some integer $m>0$), then we define the $\bar{d}$-distance between them by:
$$\bar{d}(\nu_1,\nu_2)= \inf_J \int m^{-1}\sum_{i=1}^m \rho(x_i,y_i)~dJ(x,y)$$
where the infimum is over all probability measures $J$ on $B^m \times B^m$ with marginals $\nu_1$ and $\nu_2$.

Let $(Y,\cC,\nu)$ be a standard probability space and $S:Y \to Y$ a measure-preserving automorphism. Also let $\cF \subset \cC$ be an $S$-invariant sub-$\sigma$-algebra and $\psi:Y \to B$ a measurable map. Let $\psi_1^m:Y \to B^m$ be the map 
$$\psi_1^m(y)=( \psi(Sy),\ldots, \psi(S^my)).$$
Also let $\psi^0_{-\infty}:Y \to B^\N$ be the map
$$\psi^0_{-\infty}(y) = (\psi(y),\psi(S^{-1}y), \psi(S^{-2}y),\ldots).$$

Then $(S,\psi,\nu)$ is {\bf $\cF$-conditionally very weak Bernoulli (VWB)} if whenever $y\in Y$ is random with $\textrm{Law}(y)=\nu$ then
$$\lim_{m\to\infty} \E [\bar{d}( \textrm{Law}( \psi^m_1(y)|  \cF),  \textrm{Law}( \psi^m_1(y)| \cF, \psi^0_{-\infty}) ] = 0.$$
A word about this expression is in order. $\textrm{Law}(\psi^m_1(y))$ is just the pushforward measure $(\psi^m_1)_*\mu$ (since $y$ has law $\mu$).  $\textrm{Law}(\psi^m_1(y)| \cF)$ is the distribution of $\psi^m_1(y)$ conditioned on $\cF$. In other words, it is the conditional expectation of $(\psi^m_1)_*\mu$ relative to $\cF$. Similarly, $\textrm{Law}( \psi^m_1(y)| \cF, \psi^0_{-\infty})$ is the conditional expectation of $(\psi^m_1)_*\mu$ relative to $\cF$ and the $\sigma$-algebra generated by $\psi^0_{-\infty}$. The expected value in the expression above is over $y$. 

In the special case in which $B$ is a finite set, Thouvenot proved that if $(S,\psi,\nu)$ is $\cF$-conditionally VWB then $(S,\psi,\nu)$ is $\cF$-relatively Bernoulli. The latter means there is an $S$-invariant sub-$\sigma$-algebra $\cG$ such that:
\begin{itemize}
\item $\cF$ is independent of $\cG$ (so for any $A \in \cF, A' \in \cG$, $\mu(A \cap A') = \mu(A)\mu(A')$),
\item the factor corresponding to $\cG$ is isomorphic to a Bernoulli shift. The entropy rate of this factor is necessarily equal to $h(S, \cG) = h(S, \cG \vee \cF|\cF)$,
\item the $\sigma$-algebras $\cF$ and $\cG$ generate the Borel $\sigma$-algebra of $Y$ up to sets of measure zero.
\end{itemize}
It is straightforward to generalize the proof to the case in which $B$ is an arbitrary complete metric space. Alternatively, one can use the fact that inverse limits of Bernoulli shifts are Bernoulli \cite{Or74}. 

Now let $T\in \Aut(X,\mu)$, $\phi:X \to \Prob(\Omega)$ and $\tilde{\mu}$ be as before this lemma. We set $Y=X\times \Omega^\Z$, $S = T\times \sigma$, $\nu = \tilde{\mu}$ and let $\psi: X \times \Omega^\Z \to \Omega$ be the map $\psi(x, y)=y_0$ (where $y=(y_i)_{i\in \Z}$). Also let $\cF$ be the $\sigma$-algebra generated by projection to the $X$-coordinate. It is straightforward to check that $(S,\psi,\nu)$ is $\cF$-conditionally VWB. Indeed, in this case, 
$$\textrm{Law}( \psi^m_1(y)|  \cF) =   \textrm{Law}( \psi^m_1(y)| \cF, \psi^0_{-\infty}).$$
To finish the lemma, it suffices to observe
$$h(S|\cF) = \int H(\phi(x))~d\mu(x).$$
Indeed, if $\cP_0$ is any finite partition on $\Omega$ we may let $\cP$ be the partition on $X\times \Omega^\Z$ given by: $(x,y)$ and $(x',y')$ are in the same parts of $\cP$ if and only if $y_0$ and $y'_0$ are in the same parts of $\cP_0$. Then the translates $\{S^n \cP\}_{n\in \Z}$ are independent relative to $\cF$. Therefore
$$h(S, \cP|\cF) = \int H_{\phi(x)}(\cP_0)~d\mu(x).$$
The entropy rate $h(S|\cF)$ is the supremum of $h(S,\cP|\cF)$ over all such $\cP$. This is because the translates $\{\psi \circ S^n\}_{n\in \Z}$ together with $\cF$ generate the Borel $\sigma$-algebra on $Y$. By the Monotone Convergence Theorem,
$$\int H(\phi(x))~d\mu(x) = \sup_{\cP_0} \int H_{\phi(x)}(\cP_0)~d\mu(x).$$

\end{proof}


Next, we extend the previous result to equivalence relations. Let $\cR$ be a countable ergodic pmp equivalence relation on $(X,\mu)$. Let $\phi:X \to \Prob(\Omega)$ be a Borel map. For $x\in X$, let $\kappa_x$ denote the probability measure on $\Omega^{[x]_\cR}$ given by: $\kappa_x$ is the direct product of $\phi(y)$ over all $y\in [x]_\cR$. 

Let $X_\phi$ be the set of all pairs $(x,\omega)$ with $x\in X$ and $\omega\in\Omega^{[x]_\cR}$. We endow $X_\phi$ with the smallest $\sigma$-algebra of sets which makes the maps $(x,\omega)\mapsto x$ and $(x,\omega)\mapsto\omega(\theta(x))$ measurable, for every $\theta\in [\mathcal R]$. Also we endow $X_\phi$ with the probability measure $\mu_\phi$ defined by
 $$\mu_\phi(x,\omega) = d\kappa_x(\omega)~d\mu(x).$$
 Let $\cR_\phi$ be the equivalence relation on $\cR$ given by $(x,\omega)\cR_\phi (x',\omega')$ if and only if $x \cR x'$ and $\omega=\omega'$.   We call $\cR_\phi$ the {\it inhomogeneous Bernoulli extension over $\cR$ with data $\phi$}. 

\begin{lem}\label{lem:inhomog2}
The inhomogeneous Bernoulli extension $\cR_\phi \to \cR$ is isomorphic to the Bernoulli extension of $\cR$ with base space entropy equal to $\int H(\phi(x))~d\mu(x)$. 
\end{lem}

\begin{proof}
Let $(K,\kappa)$ be a probability space with entropy equal to $\int H(\phi(x))~d\mu(x)$. Let $\theta\in [\cR]$ be an ergodic element (see \cite[Theorem 3.5]{Ke10}). By Lemma \ref{lem:inhomog}, the inhomogeneous Bernoulli shift over $\theta$ with data $\phi$ is isomorphic to $\theta \times S$, where $S$ is the Bernoulli shift with base space $(K,\kappa)$. Let $\Phi: X \times \Omega^\Z \to X \times K^\Z$ be such an isomorphism.

Given $x\in X$ and $\omega: [x]_\cR \to \Omega$, let $\omega^x \in \Omega^\Z$ be the map 
$$\omega^x(n) = \omega(\theta^n x).$$
Also, define
$\omega':[x]_\cR \to K$ by
$$\omega'(y) = \Phi(y,\omega^y)_0.$$
That is, $\omega'$ is the time 0 coordinate of $\Phi(y,\omega^y)$. Finally, define $\Psi: X_\phi \to X_K$ by 
$$\Psi(x,\omega)= (x,\omega').$$ 
The proof that $\Psi$ is the desired isomorphism is similar to the proof of Theorem \ref{thm:isomorphism}.
\end{proof}

\begin{proof}[Proof of Theorem \ref{thm:compression}]

Without loss of generality, we may assume $K$ is a compact metrizable space. 
Let $*$ be an element not contained in $K$. Let $K_*=K\cup\{*\}$ be the disjoint union and $\Omega=K_*^\N$. For $S\subset\mathbb N$, we identify the product space $K^S$ with the set of sequences $\alpha=(\alpha_1,\alpha_2,...)\in\Omega$ such that $\alpha_i\in K$, if $i\in S$, and $\alpha_i=*$, if $i\notin S$. We also view the product measure $\kappa^S$ as a measure on $\Omega$ be letting $\kappa^S(\Omega\setminus K^S)=0$.

Since $\mathcal R$ is ergodic, we can find $\theta_1,\theta_2,...\in [[\mathcal R]]$ such that dom$(\theta_i)\subset Y$, for all $i$, $\theta_i(Y)\cap\theta_j(Y)=\emptyset$, for all $i\not=j$, and $\cup_i\theta_i(Y)$ is co-null in $X$. Then $\sum_i\mu(\text{dom}(\theta_i))=1$. 

For $x\in Y$, let $S(x)$ be the set of all $i \in \N$ with $x\in\text{dom}(\theta_i)$. 
Define $\phi: Y \to \Prob(\Omega)$ by $\phi(x) =\kappa^{S(x)}$. 
Let $(X_K,\mu_\kappa)$ be the underlying space of the Bernoulli extension $\mathcal R_K$. We denote by $\tY$ the lift of $Y$ to $X_K$. 
Let $(X_\phi,\mu_\phi)$ be the underlying space of the Bernoulli extension of $\cR \resto Y$ by $\phi$. 

Define the isomorphism $\Phi:\tilde Y \to X_\phi$ by $\Phi(x,\omega) = (x, \omega')$ where $\omega': [x]_{\cR \resto Y} \to \Omega=K_*^{\mathbb N}$ is defined by $\omega'(y)_n=*$,  if $n \notin S(y)$, and $\omega'(y)_n = \omega(\theta_n(y))$, otherwise. Below, we prove that $\Phi$ is an isomorphism between the extension $\cR_K \resto\tY \to \cR \resto Y$ and the extension $(\cR \resto Y)_\phi \to \cR \resto Y$.   Lemma \ref{lem:inhomog2} implies that the extension $(\cR \resto Y)_\phi \to \cR \resto Y$ is isomorphic to the Bernoulli extension of $\cR\resto Y$ with base space entropy equal to
$$\mu(Y)^{-1}\int_Y H(\phi(x))~d\mu(x) = \mu(Y)^{-1}H(K,\kappa)\sum_{i=1}^\infty \mu(\dom(\theta_i)) = H(K,\kappa)/\mu(Y),$$
where the $\mu(Y)^{-1}$ terms appear because we have to renormalize the measure on $Y$. So this finishes the theorem.

To prove that $\Phi$ is invertible, for $(x,\omega')\in X_\phi$, define $\omega'': [x]_\cR \to K$ by
$$\omega''( \theta_n(y)) = \omega'(y)_n \textrm{ for } y \in \dom(\theta_n) \cap [x]_\cR.$$
Because $\{ \theta_n(y):~ y \in \dom(\theta_n) \cap [x]_\cR, n \in \N\}$ partitions $[x]_\cR$, this is well-defined. Define $\Psi:X_\phi \to \tY$ by
$$\Psi(x,\omega')=(x,\omega'').$$
Then $\Psi$ is the inverse of $\Phi$. 

Fix $x \in Y$. Because $\{ \theta_n(y):~ y \in \dom(\theta_n) \cap [x]_\cR, n \in \N\}$ partitions $[x]_\cR$, $\Phi$ maps the fiber measure $\kappa^{[x]_{\cR}}$ to the fiber measure $\prod_{z} \kappa^{S(z)}$ where the production is over $z \in [x]_\cR \cap \bigcup_i \dom(\theta_i)$. Therefore, $\Phi_*(\mu_\kappa\resto Y)  = \mu_\phi$.
\end{proof}


\section{Bernoulli percolation on graphed equivalence relations}

This section is devoted to the first part of the proof of main assertion of Theorem \ref{main}. 
We start by recalling several concepts and results regarding Bernoulli percolation on graphs. 
 
 \subsection{Bernoulli percolation on graphs} Let $\mathcal G=(V,E)$ be an infinite (multi-)graph with vertex set $V$ and symmetric set of edges $E$. That is, we allow multiple edges between two given vertices. A connected component of $\mathcal G$ is called a {\it cluster}. We identify points in the standard Borel space $\{0,1\}^E$ with subsets of the edge set $E$. This allows us to view $\{0,1\}^E$ as the Borel space of all subgraphs of $\mathcal G$ with the same set of vertices $V$. 
 
  A {\it simple cycle} in $\mathcal G$ is a cycle that does not use any vertex or edge more than once. A {\it simple bi-infinite path} in $\mathcal G$ is a bi-infinite path that does not use any vertex or edge more than once.
 
An infinite set of vertices $V_0\subset V$ is {\it end convergent} if for every finite $K\subset V$, there is a connected component of $\mathcal G\setminus K$ that contains all but finitely many vertices of $V_0$. Two end-convergent sets $V_0,V_1$ are {\it equivalent} if $V_0\cup V_1$ is end-convergent. An {\it end} of $\mathcal G$ is an equivalence class of end-convergent sets.

The {\it Bernoulli}($p$) {\it bond percolation} on $\mathcal G$ is the process of independently keeping edges with probability $p$ and deleting them with probability $1-p$. Concretely, we endow $\{0,1\}^E$ with the probability measure $\lambda_p^E$, where $\lambda_p$ is the probability measure on $\{0,1\}$ with weights $1-p$ and $p$.

Since the event that $\omega \in \{0,1\}^E$ has an infinite cluster is a tail event, Kolmogorov's $0$-$1$ law implies that the probability $\alpha(p):=\lambda_p^E(\{\omega\in \{0,1\}^E|\;\omega \;\text{has an infinite cluster}\})$ is equal to $0$ or $1$. 

The {\it critical value} $p_c(\mathcal G)\in [0,1]$ is defined as $p_c(\mathcal G)=\sup\{p\in [0,1]|\;\alpha(p)=0\}$. It is easy to see that if $p\geq p_c(\mathcal G)$, then \begin{equation}\label{pc}p_c(\omega)={p_c(\mathcal G)}/{p},\;\;\; \text{for}\;\;\; \lambda_p^E\text{-almost every}\;\;\; \omega\in \{0,1\}^E.\end{equation}

One also defines $p_u(\mathcal G)$ as the infimum of the set of $p\in [0,1]$ such that $\omega$ has a unique infinite cluster, for $\lambda_p^E$-almost every $\omega\in \{0,1\}^E$.

The following result due to I. Benjamini and O. Schramm  (see \cite[Theorem 4]{BS96}) provides an upper bound for $p_c$.

\begin{theorem}\label{BS96} If $\mathcal G=(V,E)$ is a graph then $p_c(\mathcal G)\leq\frac{1}{\iota(\mathcal G)+1}$. 
\end{theorem}

For any subset $A\subset \{0,1\}^E$ and  edge $e\in E$, we denote by $\Pi_eA\subset\{0,1\}^E$ the set $\{\omega\cup\{e\}|\omega\in A\}$. We also denote by $\Pi_{\neg e}A$ the set $\{\omega\setminus\{e\}|\omega\in A\}$.  

The Bernoulli$(p)$ percolation with $p\in (0,1]$ is {\it insertion tolerant}: if $A\subset\{0,1\}^E$ is a Borel subset with $\lambda_p^E(A)>0$, then $\lambda_p^E(\Pi_eA)>0$, for any edge $e\in E$. If $p\in [0,1)$ then it is {\it deletion tolerant}: if $A\subset\{0,1\}^E$ is a Borel subset with $\lambda_p^E(A)>0$, then $\lambda_p^E(\Pi_{\neg e}A)>0$, for any edge $e\in E$.
Moreover, we have that $\lambda_p^E(\Pi_eA)\geq p\lambda_p^E(A)$ and $\lambda_p^E(\Pi_{\neg e}A)\geq (1-p)\lambda_p^E(A)$.

We end this subsection with two well-known consequences of insertion and deletion tolerance:

\begin{lemma}\label{perco}
Let $\mathcal G=(V,E)$ be a multi-graph and $p\in (0,1)$. Assume that $\omega$ has $N_p$ infinite clusters, for $\lambda_p^E$-almost every $\omega\in \{0,1\}^E$, for some constant $N_p\in\mathbb N\cup\{\infty\}$. 
Then  we have

(1) If $\mathcal G$ is connected, then $N_p\in\{0,1,\infty\}$. 

(2) If $N_p=1$, then the infinite cluster of $\omega$ has one end, for $\lambda_p^E$-almost every $\omega\in \{0,1\}^E$.

\end{lemma}
{\it Proof.} Part (1) is a direct consequence of insertion tolerance and is due to Newmann and Schulman (see \cite{NS81}  and the second part of the proof of\cite[Theorem 7.6]{LP13}). For part (2), we reproduce the argument given in the proof of \cite[Theorem 7.33]{LP13}. If $\omega$ has a unique infinite cluster for almost every $\omega\in \{0,1\}^E$, then that  cluster has one end. Otherwise, by removing a finite number of edges and using deletion tolerance, we would get that $\omega$ has at least two infinite clusters with positive probability. \hfill $\square$

\subsection{Infinitely many infinite clusters} Before stating the main result of this section, we need to introduce some notation that we will use throughout this and the next section.

\begin{notation}\label{setting} 
 Let $\mathcal R$ be an ergodic countable pmp equivalence relation on a probability space $(X,\mu)$. Suppose that $\mathcal R$ is generated by finitely many automorphisms $\theta_1,...,\theta_n\in [\mathcal R]$. 

\begin{itemize}

\item For $x\in X$, we define an unoriented connected (multi-)graph $\mathcal G_{x}=([x]_{\mathcal R},E_x)$ whose edge set $E_x$ consists of the pairs $\{y,\theta_i(y)\}$ with $y\in [x]_{\mathcal R}$ and $i\in\{1,...,n\}$ (see Section \ref{isop}). 

\item Fix  $p\in (0,1)$ and endow $\{0,1\}$ with the probability measure $\lambda_p$ with weights $1-p$ and $p$. 
 
\item Let $\tilde X$ be the set of pairs $(x,\omega)$ with $x\in X$ and $\omega\in \{0,1\}^{E_x}$, and endow $\tilde X$ with the probability measure $\tilde\mu$ given by $d\tilde\mu(x,\omega)=d\lambda_p^{E_x}(\omega)d\mu(x)$. 

\item Let $\tilde{\mathcal R}$ be the equivalence relation on $\tilde X$ given by $(x,\omega)\tilde{\mathcal R}(y,\xi)$ iff  $x\mathcal Ry$ and $\omega=\xi$.   
\end{itemize}
\end{notation}

Let $u:[\mathcal R]\rightarrow\mathcal U(L^2(\mathcal R,m))$ be the unitary representation defined in section \ref{equiv}.
Consider the self-adjoint operator $T=\sum_{i=1}^n(u(\theta_i)+u(\theta_i^{-1}))$ and note that $\|T\|\leq 2n$. 
The main goal of this section is to show that if $\|T\|\leq n$, then there is a non-trivial interval of $p\in (0,1)$ such that $\omega$ has infinitely many infinite clusters, for almost every $(x,\omega)\in\tilde X$.

Here, we view every $\omega\in\{0,1\}^{E_x}$ as a subgraph of $\mathcal G_x$. Recall that we allow $\mathcal G_x$ (and therefore $\omega$) to have multiple edges joining the same two points.

\begin{theorem}\label{percolation} In the setting from above, assume that  $\frac{1}{(2n-\|T\|)+1}<p<\frac{1}{\|T\|}$.

 Then $\omega$ has infinitely many infinite clusters, for $\tilde\mu$-almost every $(x,\omega)\in\tilde X$.
\end{theorem}

\begin{remark}
Let $\mathcal G$ be the Cayley graph of a countable group $\Gamma$ with respect to a finite symmetric set of generators $S$. Let $\lambda:\Gamma\rightarrow\mathcal U(\ell^2\Gamma)$ be the left regular representation of $\Gamma$. Put $T=\sum_{g\in S}\lambda(g)$. I. Pak and T. Smirnova-Nagnibeda showed that if $\|T\|\leq \frac{|S|}{2}$, then $p_c(\mathcal G)<p_u(\mathcal G)$ (see \cite{PS-N00}). Theorem \ref{percolation} is an analogue of their result for equivalence relations.
\end{remark}

Towards Theorem \ref{percolation}, we first prove three lemmas:

\begin{lemma}\label{ber}
$\tilde{\mathcal R}$ is isomorphic to the Bernoulli extension of $\mathcal R$ with base space $(\{0,1\}^n,\lambda_p^n)$.
\end{lemma}
{\it Proof.} For  $x\in X$, the map $\beta_x:[x]_{\mathcal R}\times\{1,...,n\}\rightarrow E_x$ given by $\beta_x(y,i)=(y,\theta_i(y))$ is a bijection. Moreover, if $[x]_{\mathcal R}=[y]_{\mathcal R}$ and we identify $E_x$ and $E_y$ in the natural way, then $\beta_x\equiv\beta_y$. It follows that $\tilde{\mathcal R}$ is indeed isomorphic to the Bernoulli shift over $\mathcal R$ with base space $(\{0,1\}^n,\lambda_p^n)$. \hfill$\square$

\begin{lemma}\label{N_p} For $(x,\omega)\in\tilde X$, let $N(x,\omega)$ be the number of infinite clusters of $\omega\in\{0,1\}^{E_x}$.

Then there exists $N_p\in\{0,1,\infty\}$ such that $N(x,\omega)=N_p$, for $\tilde\mu$-almost every $(x,\omega)\in\tilde X$.
\end{lemma}

{\it Proof.} Combining lemmas  \ref{ergodic} and \ref{ber} yields that $\tilde {\mathcal R}$ is ergodic.  Since the measurable function $N:\tilde X\rightarrow\mathbb N\cup\{\infty\}$ is $\tilde{\mathcal R}$-invariant, we can find 
$N_p\in\mathbb N\cup\{\infty\}$ such that $N(x,\omega)=N_p$, for almost every $(x,\omega)\in\tilde X$. Hence, we can find $x\in X$ such that  $\omega$ has $N_p$ infinite clusters, for $\lambda_p^{E_x}$-almost every 
$\omega\in \{0,1\}^{E_x}$. Since $\mathcal G_x$ is connected, Lemma \ref{perco} (1) implies that $N_p\in\{0,1,\infty\}$. \hfill$\square$

\begin{lemma}\label{N_pp} If $\frac{1}{(2n-\|T\|)+1}<p\leq1$, then $N_p\in\{1,\infty\}$. 
\end{lemma}

{\it Proof.} By combining Lemma \ref{iso} and Theorem \ref{BS96} we get that \begin{equation}\label{p_c}p_c(\mathcal G_x)\leq \frac{1}{\iota(\mathcal G_x)+1}\leq\frac{1}{(2n-\|T\|)+1}<p,\;\;\;\text{for}\;\;\mu\text{-almost every}\;\; x\in X.\end{equation}

Therefore, for almost every $x\in X$, we have that $\omega$ has at least one infinite cluster, for $\lambda_p^{E_x}$-almost every $\omega\in \{0,1\}^{E_x}$.
Thus, $N(x,\omega)\geq 1$, for almost every $(x,\omega)\in\tilde X$. Together with Lemma \ref{N_p} this gives that $N_p\in\{1,\infty\}$.
\hfill$\square$

We are now ready to prove Theorem \ref{percolation}. The proof is an adaptation of an argument due to O. Schramm showing that $p_u(\mathcal G)\geq 1/\gamma(\mathcal G)$, for any transitive graph $\mathcal G$ (see \cite[Theorem 7.33]{LP13}). Here, $\gamma(\mathcal G):=\limsup_{n\rightarrow\infty}a_n(\mathcal G)^{1/n}$, where $a_n(\mathcal G)$ is the number of simple cycles of length $n$ in $\mathcal G$.

\subsection{Proof of Theorem \ref{percolation}} 
By Lemma \ref{N_pp} we have that  $N_p\in\{1,\infty\}$. To show that $N_p=\infty$, assume by contradiction that $N_p=1$. Thus,  $\omega$ has a unique infinite cluster, for almost every $(x,\omega)\in\tilde X$. Denote by $C(x,\omega)$ this unique infinite cluster.
Lemma \ref{perco} (2) then implies that $C(x,\omega)$ has one end, for almost every $(x,\omega)\in\tilde X$.

 Let $\mathcal A$ be the set of $(x,\omega)\in\tilde X$ such that $\omega$ (viewed again as a subgraph of $\mathcal G_x=([x]_{\mathcal R},E_x)$) contains an infinite number of simple cycles through the vertex $x$. We continue with the following:

{\bf Claim.} $\tilde\mu(\mathcal A)>0$.

{\it Proof of the claim.}
By inequality \ref{p_c} we have that $p_c(\mathcal G_x)<p$, for almost every $x\in X$. In combination with formula \ref{pc} we get that $p_c(\omega)=p_c(\mathcal G_x)/p<1$, for almost every $(x,\omega)\in\tilde X$. On the other hand,  if a graph $\mathcal G$ of bounded degree does not contain a simple bi-infinite path, then $p_c(\mathcal G)=1$ (see \cite[Lemma 3.19]{LPS06}). Altogether, we deduce that $\omega$ contains a simple bi-infinite path, for almost every $(x,\omega)\in\tilde X$.

Recall that we view $\omega$ as a graph with vertex set $[x]_{\mathcal R}$.
It follows that there is a measurable map $\theta:X\rightarrow X$  such that for almost every $x\in X$, we have that $\theta(x)\in [x]_{\mathcal R}$ and that the set of $\omega\in \{0,1\}^{E_x}$ for which there is a simple bi-infinite path in $\omega$ containing $\theta(x)$ has positive measure. Since $\mu(\theta(X))>0$ and $\mathcal G_{\theta(x)}$ is naturally identified with $\mathcal G_x$, the set $\mathcal B$ of $(x,\omega)\in\tilde X$ for which there exists a simple bi-infinite path in $\omega$ containing $x$  must also have positive measure.

Since $\omega$ has a unique infinite cluster, we derive that there is a simple bi-infinite path  in $C(x,\omega)$ containing $x$, for almost every $(x,\omega)\in\mathcal B$.
Now, we view such an infinite path as the union of two disjoint infinite simple paths starting at $x$. Since $C(x,\omega)$ has only one end, these two paths can be connected by paths in $C(x,\omega)$ that do not intersect any given finite subset of $C(x,\omega)$. This implies that there are an infinite number of simple cycles in $C(x,\omega)$  (and hence in $\omega$) through $x$, for almost every $(x,\omega)\in\mathcal B$.  
We conclude that $\tilde\mu(\mathcal A)\geq\tilde\mu(\mathcal B)>0$, which proves the claim. \hfill$\square$

Next, let $m=2n$ and enumerate $\{\psi_1,...,\psi_m\}=\{\theta_1,...,\theta_n,\theta_1^{-1},...,\theta_n^{-1}\}$. Note that for every $y\in [x]_{\mathcal R}$ there are exactly $m$ edges having $y$ as an endpoint, namely $(y,\psi_i(y))$, for $i\in\{1,...,m\}$.

For $k\geq 1$ and $i_1,...,i_k\in\{1,...,m\}$, we define $A_{i_1,...,i_k}$ to be the set of $x\in X$ such that $\psi_{i_k}...\psi_{i_2}\psi_{i_1}(x)=x$ and $x\not=\psi_{i_a}...\psi_{i_1}(x)\not=\psi_{i_b}...\psi_{i_1}(x)\not=x$, for all $1\leq a<b<k$. In this case, $x,\psi_{i_1}(x),...,\psi_{i_k}...\psi_{i_1}(x)$ is a simple cycle in $\mathcal G_x$. Conversely, any simple cycle in $\mathcal G_x$ containing $x$ is of this form. Further, we define $\tilde A_{i_1,...,i_k}$ to be the measurable set of $(x,\omega)\in\tilde X$ such that $x\in A_{i_1,...,i_k}$ and the cycle  $x,\psi_{i_1}(x),...,\psi_{i_k}...\psi_{i_1}(x)$ belongs to $\omega$. 

Then $\mathcal A$ consists of the points $(x,\omega)\in\tilde X$ which belong to infinitely many sets of the form $\tilde A_{i_1,...,i_k}$. Since $\tilde\mu(\mathcal A)>0$ by the claim, we derive that $\sum_{k=1}^{\infty}\sum_{i_1,...,i_k\in\{1,...,m\}}\tilde\mu(\tilde A_{i_1,...,i_k})=\infty.$ Since $\tilde\mu(\tilde A_{i_1,...,i_k})=p^k\mu(A_{i_1,...,i_k})$ we conclude that \begin{equation}\label{cycles}\sum_{k=1}^{\infty}p^k\;\left(\sum_{i_1,...,i_k\in\{1,...,m\}}\;\mu(A_{i_1,...,i_k})\right)=\infty.\end{equation}

Let $\Delta=\{(x,x)|x\in X\}$ and view ${\bf 1}_{\Delta}\in L^2(\mathcal R,m)$. Then $\langle u(\psi)({\bf 1}_{\Delta}),{\bf 1}_{\Delta}\rangle=\mu(\{x\in X|\psi(x)=x\})$, for every $\psi\in [\mathcal R]$. Hence, since $T=\sum_{i=1}^mu(\psi_i)$, for every $k\geq 1$ we have that \begin{equation}\label{T} \langle T^k({\bf 1}_{\Delta}),{\bf 1}_{\Delta}\rangle=\sum_{i_1,...,i_k\in\{1,...,m\}}\mu(\{(x\in X|\psi_{i_k}...\psi_{i_1}(x)=x\})\geq\sum_{i_1,...,i_k\in\{1,...,m\}}\mu(A_{i_1,...,i_k}).\end{equation}

By combining equations \ref{cycles} and \ref{T} we deduce that $\sum_{k=1}^{\infty}p^k\langle T^k({\bf 1}_{\Delta}),{\bf 1}_{\Delta}\rangle=\infty$. This implies that $p\|T\|\geq 1$ which leads to the desired contradiction.
\hfill$\square$


\section{Ergodicity of the cluster equivalence relation}\label{erg}

This section is devoted to the second part of the proof of the main assertion of Theorem \ref{main}.

Consider the setting from \ref{setting}. In particular, $p\in (0,1)$ is fixed, and $\tilde X$ is the set of pairs $(x,\omega)$, with $x\in X$ and $\omega\in\{0,1\}^{E_x}$, endowed with the probability measure given by $d\tilde\mu(x,\omega)=d\lambda_p^{E_x}(\omega)d\mu(x)$.  Two points $(x,\omega),(y, \xi)\in\tilde X$ are $\tilde{\mathcal R}$-equivalent if $x\mathcal Ry$ and $\omega=\xi$. Recall that we view every $\omega\in\{0,1\}^{E_x}$ as a subgraph of $\mathcal G_x=([x]_{\mathcal R},E_x)$.

Following D. Gaboriau  \cite[Section 1.2]{Ga05} we define a subequivalence relation $\tilde{\mathcal R}_{\text{cl}}$ of $\tilde{\mathcal R}$, called the {\it cluster equivalence relation}. 
Thus, we say that  two points $(x,\omega),(y, \xi)\in\tilde X$ are $\tilde{\mathcal R}_{\text{cl}}$-equivalent if they are $\tilde{\mathcal R}$-equivalent and $x,y$ belong to the same cluster of $\omega=\xi$. 

For $(x,\omega)\in\tilde X$, we let $C(x,\omega)$ be the cluster of $x$ in $\omega$.
We denote by $U^{\infty}$ the set of points $(x,\omega)\in\tilde X$ such that $C(x,\omega)$ is infinite. Then $U^{\infty}$ is an $\tilde{\mathcal R}_{\text{cl}}$-invariant set and the restriction 
${\tilde{\mathcal R}_{\text{cl}}} {\resto U^{\infty}}$ has infinite classes.

In this section we show that if $\omega$ has infinitely many infinite clusters, for almost every $(x,\omega)\in\tilde X$, then ${\tilde{\mathcal R}_{\text{cl}}} {\resto U^{\infty}}$ is ergodic and has cost $>1$. 

\begin{theorem}\label{ergo}
Assume that $\omega$ has infinitely many infinite clusters, for $\tilde\mu$-almost every $(x,\omega)\in\tilde X$.

Then the restriction ${\tilde{\mathcal R}_{\text{cl}}} {\resto U^{\infty}}$ is ergodic. 
\end{theorem}

R. Lyons and O. Schramm proved that the infinite clusters that may appear in Bernoulli$(p)$ bond percolation on a transitive graph are {\it indistinguishable}  (see \cite[Theorem 1.1]{LS99} for the precise statement).
D. Gaboriau and R. Lyons then showed that indistinguishability of infinite clusters is equivalent to ergodicity of the restriction of the cluster equivalence relation to its infinite locus (see \cite[Proposition 5]{GL07}). Theorem \ref{ergo} is a generalization of these results.
Its proof is an immediate consequence of  work of D. Aldous and R. Lyons \cite{AL06} who noted that the results from \cite{LS99} extend to the more general context of unimodular random networks. 
More precisely, we will show that the following result, stated implicitly in \cite{AL06}, implies Theorem \ref{ergo}.

\begin{theorem}\label{indis}  Let $\mathcal A$ be a Borel subset of the set $\{(A,x)| A\in\{0,1\}^{[x]_{\mathcal R}}\times\{0,1\}^{E_x}, x\in X\}$.
Assume that if $(A,x)\in \mathcal A$ and $y\in [x]_{\mathcal R}$, then $(A,y)\in\mathcal A$.

Then the set of $(x,\omega)\in\tilde X$, for which there exist two infinite clusters $C_1,C_2$ of $\omega$ such that $((C_1,\omega),x)\in\mathcal A$ and $((C_2,\omega),x)\notin\mathcal A$, has $\tilde\mu$-measure zero.
\end{theorem} 

Before deducing Theorem \ref{ergo} from Theorem \ref{indis}, let us   explain how the latter follows from \cite{AL06}. Recall from \cite[Section 2]{AL06} that a {\it network} is a (multi-)graph $\mathcal G=(V,E)$ together with a complete separable metric space $\Xi$ and maps from $V$ and $E$ to $\Xi$. A {\it rooted network} $(\mathcal G,o)$ is a network  with a distinguished vertex $o$. Then $\mathcal G_*$ denotes the set of  isomorphism classes of rooted connected locally finite networks. 

By \cite[Example 9.9]{AL06} the graphs $(\mathcal G_x)_{x\in X}$ give rise to a unimodular random rooted network. More precisely, consider the map $\Phi:X\rightarrow\mathcal G_*$ given by $\Phi(x)=(\mathcal G_x,x)$. Then the push-forward $\Phi_*\mu$ is a unimodular probability measure on $\mathcal G_*$ (see \cite[Definition 2.1]{AL06}). Moreover, the measure $\tilde\mu$ corresponds to  Bernoulli($p$) percolation on $\Phi_*\mu$. Since $p\in (0,1]$, we have that $\tilde\mu$  is insertion tolerant in the sense of \cite[Definition 6.4]{AL06}. Therefore, by \cite[Theorem 6.15]{AL06},  $\tilde\mu$ has indistinguishable infinite clusters. Finally, translating this fact leads to Theorem \ref{indis}.

{\it Proof of Theorem \ref{ergo}}.   Let $Y\subset U^{\infty}$ be a ${{\tilde{\mathcal R}}_{\text{cl}}}$-invariant Borel subset. 
We define  $\mathcal A$ as the set of $((C,\omega),x)$ with $x\in X$, $\omega\in\{0,1\}^{E_x}$ and $C$ infinite cluster of $\omega$ such that $(y,\omega)\in Y$, for all $y\in C$. 

Let $x\in X$, $\omega\in\{0,1\}^{E_x}$ and $C$ infinite cluster of $\omega$ such that  $((C,\omega),x)\not\in\mathcal A$. Then 
 $(y,\omega)\notin Y$, for some $y\in C$. But then for all $z\in C$ we have that $(z,\omega)\sim_{{{\tilde{\mathcal R}}_{\text{cl}}}}(y,\omega)$ and since $Y$ is   ${{\tilde{\mathcal R}}_{\text{cl}}}$-invariant, we deduce that $(z,\omega)\notin Y$.

Since $\mathcal A$ is clearly invariant under changing the ``root" $x$, Theorem \ref{indis} implies that  for almost every $(x,\omega)\in\tilde X$ we have that either $(y,\omega)\in Y$, for all $y$ contained in some infinite cluster of $\omega$, or  $(y,\omega)\notin Y$, for all $y$ contained in some infinite cluster of $\omega$.

This  implies that $Y$ is invariant under ${\tilde{\mathcal R}}{\resto U^{\infty}}$. Since by lemmas  \ref{ergodic} and \ref{ber} we have that ${\tilde{\mathcal R}}$ is ergodic, it follows that $\tilde\mu(Y)\in\{0,\tilde\mu(U^{\infty})\}$, which proves  that ${{\tilde{\mathcal R}}_{\text{cl}}}{\resto U^{\infty}}$  is ergodic.\hfill$\square$

\begin{proposition}\label{transient}
Assume $\omega$ has infinitely many infinite clusters, for $\tilde\mu$-almost every $(x,\omega)\in\tilde X$.
Then  the normalized cost of ${\tilde{\mathcal R}_{\text{cl}}} {\resto U^{\infty}}$ is $>1$.
\end{proposition}

The proposition follows by combining Theorem \ref{ergo}, \cite{NS81}, and \cite[Corollaire IV.24]{Ga99}. For the reader's convenience we include a proof below.

{\it Proof.}  We begin with the following claim:

{\bf Claim.} Each infinite cluster of $\omega$ has infinitely many ends, for $\tilde\mu$-almost every $(x,\omega)\in\tilde X$.

{\it Proof of the claim.} The proof  is a straightforward adaptation of  the proofs of Propositions 3.9 and 3.10 in \cite{LS99}. 
By the discussion following \cite[Conjecture 4.1]{BS96} it is enough to show that no infinite cluster of $\omega$ has an isolated end. Assume that some cluster of $\omega$ has an isolated end, with positive probability. Then insertion tolerance guarantees that, with positive probability, a cluster of $\omega$ will have at least $3$ ends with one of them being isolated. 

Let $A_n$ be the set of $(x,\omega)\in\tilde X$ with the property that $C(x,\omega)\setminus\{y\in C(x,\omega)|\;d(x,y)\leq n\}$ has at least $3$ infinite components, where $d$ is the cluster metric. Our assumption implies  that the set of $(x,\omega)\in A_n$ for which $C(x,\omega)$ has an isolated end, has positive probability, for some $n\geq 1$.

If $C(x,\omega)\cap A_n\not=\emptyset$,  then we let $K(x,\omega)$ be the set of $y\in C(x,\omega)\cap A_n$ that are closest to $x$. Next, we let $\tilde m$ be the usual infinite measure of $\tilde{\mathcal R}$ and define $F:\tilde{\mathcal R}\rightarrow [0,1]$ by letting $$F((x,\omega),(y,\omega))=\begin{cases}|K(x,\omega)|^{-1}\;\;\text{if}\;\;C(x,\omega)\cap A_n\not=\emptyset\;\;\text{and}\;\;y\in K(x,\omega)\\0\;\;\text{otherwise}\end{cases}$$

Since $\sum_{(y,\omega)\in [(x,\omega)]_{\tilde{\mathcal R}}}F((x,\omega),(y,\omega))\in\{0,1\},$ for all $(x,\omega)\in\tilde X$, we get that $\int_{\tilde R}F\;\text{d}\tilde m\leq 1$.

 On the other hand, let $(x,\omega)\in A_n$ and $\eta$ be an isolated end of $C(x,\omega)$. Then we can find $B\subset C(x,\omega)$ finite and a neighborhood $D$ of $\eta$ such that the points in $C(x,\omega)\cap A_n$ that are closest to any given point $y\in D$ lie in $B$.
Thus, we have that $K(y,\omega)\subset B$, for all $(y,\omega)\in[(x,\omega)]_{\tilde{\mathcal R}}$ with $y\in D$. In particular, $|K(y,\omega)|\leq |B|$,  for all such $y$.

Since $D$ is infinite, it follows that $\sum_{(y,\omega), (z,\omega)\in [(x,\omega)]_{\tilde{\mathcal R}}, z\in B}F((y,\omega),(z,\omega))=\infty$. Since $B$ is finite, we derive that $\sum_{(y,\omega)\in [(x,\omega)]_{\tilde{\mathcal R}}}F((y,\omega),(z,\omega))=\infty$, for some $z\in B$.  This clearly implies that $\int_{\tilde{\mathcal R}}F\;\text{d}\tilde m=\infty$, which gives a contradiction.\hfill$\square$

 For $i\in\{1,...,n\}$, let $A_i$ be the set of $(x,\omega)\in U^{\infty}$ such that  $x$ and $\theta_i(x)$ lie in the same cluster of $\omega$. We define
$\tilde\theta_i\in[[\tilde{\mathcal R}_{\text{cl}}]]$ by letting $\tilde{\theta_i}(x,\omega)=(\theta_i(x),\omega)$, for all $(x,\omega)\in A_i$.
Then $\{\tilde{\theta_i}\}_{i=1}^n$ is a generating graphing of $\tilde{\mathcal R}_{\text{cl}}\resto {U^{\infty}}$. Moreover,  for all $(x,\omega)\in U^{\infty}$, the graph of the equivalence class of $(x,\omega)$ in $\tilde{\mathcal R}_{\text{cl}}\resto U^{\infty}$ associated to $\{\tilde{\theta_i}\}_{i=1}^n$ is isomorphic to the cluster $C(x,\omega)$. 

By the claim, the latter has infinitely many ends, for almost every $(x,\omega)\in U^{\infty}$. 
Since ${\tilde{\mathcal R}_{\text{cl}}}{\resto U^{\infty}}$  is ergodic by Theorem \ref{ergo},   \cite[Corollaire IV.24]{Ga99} gives that ${\tilde{\mathcal R}_{\text{cl}}}{\resto U^{\infty}}$  has normalized cost $>1$. 
 \hfill$\square$
 


\section{Proofs of Theorem \ref{main} and Corollary \ref{cor}}

\subsection{A generalization of Theorem \ref{main}} 
The main goal of this section is to prove Theorem \ref{main}.
Let $\cR$ be a non-amenable countable ergodic pmp equivalence relation on a probability space $(X,\mu)$.  We would like to understand for which probability spaces $(K,\kappa)$ there exist a free ergodic pmp action $\mathbb F_2\curvearrowright (X_K,\mu_\kappa)$ such that $\mathcal R(\mathbb F_2\curvearrowright X_K)\leq \cR_K$, almost everywhere. While we expect that this should be the case for any non-trivial $(K,\kappa)$, at this point we only have partial answers. The next theorem, which clearly generalizes Theorem \ref{main}, summarizes our main results. Recall the definition of the Shannon entropy $H(K,\kappa)$ from \S \ref{sec:isom}.

\begin{thm}\label{thm:other}
Let $\cR$ be a non-amenable countable ergodic pmp equivalence relation on $(X,\mu)$.
\begin{enumerate}
\item There is a number $\beta(\cR) \in [0,\infty]$ such that if $H(K,\kappa)>\beta(\cR)$, then there exists a free ergodic pmp action $\mathbb F_2\curvearrowright (X_K,\mu_\kappa)$ such that $\mathcal R(\mathbb F_2\curvearrowright X_K)\leq \cR_K,$ almost everywhere. If $H(K,\kappa)<\beta(\cR)$, then no such action exists. 

\item $\beta(\cR)$ is finite. In particular, if $(K,\kappa)$ is non-atomic, then there exists a free ergodic pmp action $\mathbb F_2\curvearrowright (X_K,\mu_\kappa)$ such that $\mathcal R(\mathbb F_2\curvearrowright X_K)\leq \cR_K,$ almost everywhere.

\item For any ergodic non-amenable subequivalence relation $\cS \le \cR$, $\beta(\cR) \le [\cR:\cS]^{-1}\beta(\cS)$. In particular, if $\cS$ has infinite index, then $\beta(\cR)=0$. 

\item For any non-null Borel set $Y \subset X$, $\beta(\cR) \le \mu(Y) \beta(\cR \resto Y)$. In particular, if $\cR$ has infinite fundamental group, then $\beta(\cR)=0$. 

\item If $\cR$ contains a normal ergodic subequivalence relation $\cS\vartriangleleft \cR$ with infinite index, then $\beta(\cR)=0$. 

\end{enumerate}
\end{thm}

Recall that the {\it fundamental group} of $\cR$ denotes  the set of all quotients $\mu(Y_1)/\mu(Y_2)\in (0,\infty)$, where $Y_1, Y_2\subset X$ are non-null Borel subsets such that $\cR \resto Y_1\cong \cR \resto Y_2$. For the definition of normality for subequivalence relations $\cS\le\cR$, see \cite{FSZ89}.

\subsection{Proof of Theorem \ref{thm:other}} We begin by defining $\beta(\mathcal R)$ and showing that it is finite.

\begin{definition}
We define $\beta(\cR) \in [0,\infty]$ to be the infimum of all numbers of the form $H(K,\kappa)$ where $(K,\kappa)$ is a probability space satisfying:  there exist a free ergodic pmp action $\mathbb F_2\curvearrowright (X_K, \mu_\kappa)$ such that $\mathcal R(\mathbb F_2\curvearrowright X_K)\leq \cR_K,$ almost everywhere.
\end{definition}

\begin{prop}\label{lem:1}
If $(L,\lambda)$ is {\em any} probability space with $H(L,\lambda)>\beta(\cR)$, then there exists a free ergodic pmp action $\mathbb F_2\curvearrowright (X_L, \mu_\lambda)$ such that $\mathcal R(\mathbb F_2\curvearrowright X_L)\leq \cR_L,$ almost everywhere.
\end{prop}

\begin{proof}
By hypothesis, there exists a probability space $(K,\kappa)$ with $H(K,\kappa)<H(L,\lambda)$ and a free ergodic pmp action $\mathbb F_2\curvearrowright (X_K, \mu_\kappa)$ such that $\mathcal S:=\mathcal R(\mathbb F_2\curvearrowright X_K)\leq \cR_K,$ almost everywhere. Let $(N,\eta)$ be a probability space such that $H(N,\eta) = H(L,\lambda) - H(K,\kappa)$. The Shannon entropy of $(N\times K,\eta \times \kappa)$ equals the Shannon entropy of $(L,\lambda)$. Theorem \ref{thm:isomorphism} implies that the extension $\cR_{L} \to \cR$ is isomorphic to $\cR_{N\times K} \to \cR$. The latter extension has $\cR_{K}$ as an intermediate factor. 

Next, we lift the action $\mathbb F_2\curvearrowright X_K$ to a free pmp action  $\mathbb F_2\curvearrowright X_{N\times K}$ so that $\tilde{\mathcal S}:=\mathcal R(\mathbb F_2\curvearrowright X_{N\times K}$) is the lift of $\mathcal S$ through the extension $\mathcal R_{N\times K}\to\mathcal R_{K}$. Since the extension $\mathcal R_{N\times K}\to\cR_{K}$ is isomorphic to the Bernoulli extension $(\mathcal R_{K})_{N}\to\mathcal R_K$,  Theorem \ref{thm:subeq} implies that the extension $\tilde{\mathcal S}\to\mathcal S$ is isomorphic to a Bernoulli extension. Since $\mathcal S$ is ergodic, Lemma \ref{ergodic} implies that $\tilde{\mathcal S}$ is ergodic, hence the action $\mathbb F_2\curvearrowright X_{N\times K}$ is ergodic. Since $\mathcal R_L\cong\mathcal R_{N\times K}$ by Theorem \ref{thm:isomorphism}, we are done. 
\end{proof}


Next, we obtain a nontrivial upper bound on $\beta$. Define $\alpha(\cR)=\log(n)$, where $n\ge 3$ is the smallest natural number  such that there exist $\theta_1,\ldots, \theta_n \in [\cR]$ with
$$\left\| \frac{1}{n} \sum_{i=1}^n u(\theta_i) \right\| < 1/4.$$

\begin{prop}\label{lem:2}
There is a universal constant $C>0$ such that
$$\beta(\cR) \le \alpha(\cR)+C.$$
In particular, $\beta(\cR)$ is finite.
\end{prop}

\begin{proof}
By Lemma \ref{spec}, non-amenability of $\cR$ implies $\alpha(\cR)$ is finite. Let $n\ge 3$ with $\log(n)=\alpha(\cR)$. Let $\theta_0 \in [\cR]$ be ergodic and $\theta_1,\ldots, \theta_n \in [\cR]$ such that
$\left\| \frac{1}{n} \sum_{i=1}^n u(\theta_i) \right\| < 1/4.$
Then
$$\left\| \frac{1}{n+1} \sum_{i=0}^n u(\theta_i) \right\| < 1/2$$
and the subequivalence relation $\cR_0$ generated by $\theta_0,\ldots, \theta_n$ is ergodic. 

Let 
$T= \sum_{i=0}^n (u(\theta_i) + u(\theta_i^{-1})).$
Then $\|T\|< n+1$. Let $p=1/(n+2)$. Note that
\begin{equation}\label{manyends}\frac{1}{2(n+1) - \|T\| +1} < p < \frac{1}{\|T\|}.\end{equation}

Consider the notation from \ref{setting}, for the ergodic equivalence relation $\mathcal R_0$ and its generating graphing $\theta_0,...,\theta_n$ (instead of $\mathcal R$ and $\theta_1,...,\theta_n$) and for the parameter $p$ defined above.

By inequality \ref{manyends}, Theorem \ref{percolation} implies that $\omega$ has infinitely many infinite clusters, for almost every $(x,\omega)\in\tilde X$.
Let ${\tilde{\mathcal R}}_{\text{cl}}\subset{\tilde{\mathcal R}}$ be the cluster  equivalence relation and $U^{\infty}\subset\tilde X$ as defined in the beginning of Section \ref{erg}.
By combining Theorem \ref{ergo} and Proposition \ref{transient} we conclude that ${\tilde{\mathcal R}}_{\text{cl}}\resto{U^{\infty}}\subset{\tilde{\mathcal R}}\resto{U^{\infty}}$ is ergodic and has normalized cost $>1$.
Moreover,  the cost of ${{\tilde{\mathcal R}}_{\text{cl}}\resto{U^{\infty}}}\subset{\tilde{\mathcal R}}\resto{U^{\infty}}$ is clearly finite. 

By Lemma \ref{ber}, $\tilde{\mathcal R}$ is isomorphic to the Bernoulli extension with base space $(K,\kappa):=(\{0,1\}^{n+1},\lambda_p^{n+1})$.
In particular, since $\mathcal R$ is ergodic, Lemma \ref{ergodic} gives that $\tilde{\mathcal R}$ is ergodic.
Therefore, we can find an ergodic subequivalence relation $\mathcal S\subset\tilde{\mathcal R}$ whose restriction to $U^{\infty}$ coincides with  ${\tilde{\mathcal R}}_{\text{cl}}\resto{U^{\infty}}$. Then the induction formula \cite[Proposition II.6 (2)]{Ga99} implies that $\mathcal S$ has cost in $(1,\infty)$. 

By applying Theorem \ref{hj} to $\mathcal S$, it follows that there exists a free ergodic pmp action $\mathbb F_2\curvearrowright (\tilde X,\tilde\mu)$ such that $\mathcal S_0:=\mathcal R(\mathbb F_2\curvearrowright\tilde X)\leq\mathcal S$. 
In particular, $\mathcal S_0\leq\tilde{\mathcal R}\cong\mathcal R_K$.
Thus, we deduce that
\begin{eqnarray*}
\beta(\cR) &\le& H(K,\kappa) =-(n+1)(p\log(p)+(1-p)\log(1-p))\\
&=& (n+1) \frac{\log(n+2)}{n+2} - (n+1)^2 \frac{\log(1-1/(n+2)) }{n+2 }\\
&\le& \log(n+2) + 1\le \log(n) + C = \alpha(\cR)+C,
\end{eqnarray*}
where $C = 1 + \log(5/3)$. 
\end{proof}

\begin{prop}\label{lem:3}
If $\cS\le\cR$ is an ergodic non-amenable 
subequivalence relation, then we have $\beta(\cR) \le \beta(\cS) [\cR:\cS]^{-1}.$
\end{prop}

\begin{proof}
Let $(K,\kappa)$ be a probability space with $H(K,\kappa)>\beta(\cS)[\cR:\cS]^{-1}$. By Theorem \ref{thm:subeq}, if $\tcS$ is the lift of $\cS$ to $\cR_K$ then $\tcS \to \cS$ is isomorphic to the Bernoulli extension of $\cS$ with base entropy $H(K,\kappa)[\cR:\cS] > \beta(\cS)$. By the definition of $\beta$, there is a free ergodic pmp action $\mathbb F_2\curvearrowright (X_K,\mu_{\kappa})$ whose orbits are contained in $\tcS$. Since $\tcS\le \cR_K$, these orbits are also contained in $\cR_K$. Therefore, $\beta(\cR) \le H(K,\kappa)$, and the inequality follows by taking the infimum over all such $H(K,\kappa)$. 
\end{proof}

\begin{prop}\label{lem:4}
Let $Y \subset X$ be a non-null Borel set. Then $\beta(\cR) \le \beta(\cR \resto Y) \mu(Y).$
\end{prop}

\begin{proof}
Let $(K,\kappa)$ be a probability space and suppose $H(K,\kappa)>\beta(\cR \resto Y) \mu(Y)$. By Theorem \ref{thm:compression}, if $\tY$ is the lift of $Y$ to $X_K$, then $\cR_K\resto \tY \to \cR \resto Y$ is isomorphic to the Bernoulli extension of $\cR\resto Y$ with base entropy $H(K,\kappa)\mu(Y)^{-1} > \beta(\cR\resto Y)$. So by the definition of $\beta$, there is a free ergodic pmp action $\mathbb F_2\curvearrowright\tY$ such that $\cS=\mathcal R(\mathbb F_2\curvearrowright\tY)$ satisfies $\mathcal S\leq\cR_K \resto \tY$, almost everywhere. 

Since $\mathcal R_K$ is ergodic by Lemma \ref{ergodic}, we can find an ergodic subequivalence equivalence relation $\mathcal T\leq\cR_K$  such that $\mathcal T\resto\tY=\mathcal S$. 
Then \cite[Theorem IV.15]{Ga99} and \cite[Proposition II.6 (2)]{Ga99} together imply that the cost of $\mathcal T$ belongs to $(1,+\infty)$. 
Theorem \ref{hj} further implies that $\mathcal T$ and thus $\mathcal R_K$ contains almost every orbit of a free ergodic pmp action $\mathbb F_2\curvearrowright X_K$. 
 Therefore, $\beta(\cR) \le H(K,\kappa)$, and the conclusion follows by taking the infimum over all such $H(K,\kappa)$. 
\end{proof}


\begin{prop}\label{thm:5}
If $\cR$ contains an ergodic normal subequivalence relation $\cN\vartriangleleft\cR$ such that $\cR/\cN$ is non-amenable, then $\beta(\cR)=0$.
\end{prop}

Recall from \cite{FSZ89} that there exists a countable group, denoted $\cR/\cN$, and a cocycle $c:\cR \to \cR/\cN$, such that $\cN$ is the kernel of $c$. Moreover, for any $\theta \in \cR/\cN$ there is an element $\ttheta \in [\cR]$ such that $c(\ttheta x,x) = \theta$ for a.e. $x$. The element $\ttheta$ is called a {\em lift} of $\theta$. These are all the facts we will need about normal subequivalence relations. We will prove Proposition \ref{thm:5} by lifting an appropriate set of elements from $\cR/\cN$ and using the bound in Proposition \ref{lem:2}. 

\begin{lem}\label{lem:3.2}
Let $\theta_1,\ldots, \theta_n \in \cR/\cN$ and let $\ttheta_1,\ldots, \ttheta_n \in [\cR]$ be lifts. Then
$$\left\| \frac{1}{n} \sum_{i=1}^n u(\ttheta_i) \right\| \le \left\|\frac{1}{n} \sum_{i=1}^n \lambda( \theta_i) \right\|$$
where $\lambda: \cR/\cN \to U(\ell^2(\cR/\cN))$ is the left-regular representation.
\end{lem}

\begin{proof}
Let $\Delta=\{(x,x)|x\in X\}$ and view ${\bf 1}_{\Delta}\in L^2(\mathcal R,m)$.
Let $\delta_e\in\ell^2(\cR/\cN)$ denote the Dirac function at the identity $e\in\cR/\cN$. Then we have 

$$\left\langle \left(\frac{1}{n} \sum_{i=1}^n u(\ttheta_i)\right) {\bf 1}_{\Delta},{\bf 1}_{\Delta}\right\rangle=\frac{1}{n}\sum_{i=1}^n\mu(\{x\in X|\ttheta_i(x)=x\})\leq\frac{1}{n}\sum_{i=1}\delta_{\theta_i,e}=\left\langle\left(\frac{1}{n}\sum_{i=1}^nu(\theta_i)\right)\delta_e,\delta_e\right\rangle.$$

The conclusion follows immediately by combining this inequality with the following three facts:
\begin{itemize}
\item If $\ttheta_1,\ttheta_2\in [\mathcal R]$ are lifts of $\theta_1,\theta_2\in\cR/\cN$, then $\ttheta_1^{-1}$ is a lift of $\theta_1^{-1}$, and $\ttheta_1\ttheta_2$ is a lift of $\theta_1\theta_2$.
\item $\|T\|=\lim\limits_{m\rightarrow\infty}\Big(\big\langle(T^*T)^m {\bf 1}_{\Delta},{\bf 1}_{\Delta}\big\rangle\Big)^{\frac{1}{2m}}$, for every $T\in L(\mathcal R)$.
\item $\Big(\big\langle(T^*T)^m \delta_e,\delta_e\big\rangle\Big)^{\frac{1}{2m}}\leq \|T\|$, for every $T\in L(\mathcal R/\mathcal N)$ and all $m\geq 1$.
\end{itemize}
\end{proof}

Let $\cF \le \cR$ be a finite subequivalence relation. 
We denote by $X/\cF$ the quotient space and by $\cR/\cF$ the quotient equivalence relation on $X/\cF$. 
More precisely, the elements of $X/\cF$ are the $\cF$-classes of $X$. 
Let $\pi:X \to X/\cF$  be the natural projection map and endow $X/\cF$ with
 the push forward measure $\mu_\cF:=\pi_*\mu$. Note that $( [x]_\cF, [y]_\cF) \in \cR/\cF$ if and only if $x\cR y$. 

We leave the proof of the following easy lemmas as exercises.

\begin{lem}\label{lem:3.3}
If $\cN\vartriangleleft \cR$ is a normal subequivalence relation and $\cF \le \cN$ is a finite subequivalence relation, then $\cN/\cF$ is normal in $\cR/\cF$. Moreover $\cR/\cN \cong (\cR/\cF)/(\cN/\cF)$.
\end{lem}


\begin{lem}\label{lem:3.4}
There exists a Borel set $Y \subset X$  such that every $\cF$-class contains exactly one element of $Y$. Moreover $\cR \resto Y \cong \cR/\cF$. 
If each $\cF$ class contains exactly $m\in \N$ elements, then $\mu(Y)=1/m$.
\end{lem}



\begin{proof}[Proof of Proposition \ref{thm:5}]

By Kesten's Theorem \cite{Ke59} non-amenability of the group $\cR/\cN$ implies the existence of elements $\theta_1,\ldots, \theta_n \in \cR/\cN$ with $n\ge 3$ such that
$$\left\| \frac{1}{n} \sum_{i=1}^n \theta_i \right\| < 1/4.$$

Let $m>1$ be a natural number. Let $\cF\le \cN$ be a finite subequivalence relation such that every $\cF$-class contains $m$ elements. By Lemma \ref{lem:3.3}, $\cR/\cN \cong (\cR/\cF)/(\cN/\cF)$. So there exist elements $\theta'_1,\ldots, \theta'_n \in (\cR/\cF)/(\cN/\cF)$ such that
$$\left\| \frac{1}{n} \sum_{i=1}^n \theta'_i \right\| < 1/4.$$
By Lemma \ref{lem:3.2} we get that $\alpha(\cR/\cF) \le \log(n)$. Lemma \ref{lem:3.4} implies that $\alpha(\cR\resto Y_m) \le 
\log(n)$, where $Y_m \subset X$ is any Borel subset with $\mu(Y_m)=1/m$. By Propositions \ref{lem:2} and \ref{lem:4}, 
$$\beta(\cR) \le \beta(\cR \resto Y_m)/m \le \log(n)/m + C/m$$
where $C>0$ is a universal constant. Taking $m \to \infty$, we obtain $\beta(\cR)=0$.
\end{proof}

By collecting the above results, we are now ready to prove Theorem \ref{thm:other}.

\begin{proof}[Proof of Theorem \ref{thm:other}]
Items (1-4) are proven in Propositions \ref{lem:1}, \ref{lem:2}, \ref{lem:3}, \ref{lem:4} respectively. To prove item (5), suppose $\cN\le \cR$ is ergodic and normal, and $\cR/\cN$ is infinite. If $\cR/\cN$ is amenable, then since $\cR$ is non-amenable, $\cN$ must also be non-amenable. In this case, the conclusion follows from Proposition \ref{lem:3}. On the other hand, if $\cR/\cN$ is non-amenable, the conclusion follows from Proposition \ref{thm:5}.
\end{proof}

\subsection{Proof of Corollary \ref{cor}} If $\mathcal{R}$ is non-amenable, it admits such an extension by Theorem \ref{main}. On the other hand, if $\mathcal{R}$ is amenable, it is hyperfinite by \cite{CFW81}. Any extension of $\mathcal{R}$ is then also hyperfinite, since the lift as in Remark \ref{R: lift} of a finite subequivalence relation remains finite. 
Thus if $\mathcal{R}$ is amenable, no extension of $\mathcal{R}$ can contain the orbit equivalence relation of a free ergodic pmp action of $\mathbb{F}_2$, as the latter is non-amenable.


 \section{Uncountably many ergodic extensions of nonamenable $\mathcal{R}$}
 
 The goal of this section is to prove Theorem \ref{T: ext}. To this end, we will make use of I. Epstein's co-induction construction \cite{Ep07}. 
 
 \subsection{Co-induced Equivalence Relation}
 
\begin{Hoff_Commands}

Let $\Gamma_0 \on{\beta} (X, \mu)$ be a free ergodic pmp action and $\R$ an ergodic pmp equivalence relation on $(X, \mu)$ such that $\R_0 = \R(\Gamma_0 \on{\beta} X) \leq \R$. 
Since $\R$ is ergodic, there is $N_0 \in \ZZ_{> 0} \cup \{\infty\}$ such that $[x]_\R$ contains exactly $N_0$ $\R_0$-classes for almost every $x \in X$. Let $N = [0, N_0) \cap \ZZ$.

Then for any pmp action $\Gamma_0 \on{\alpha} (Y, \nu)$, there is a pmp countable equivalence relation $\R_\alpha = {\rm CInd}_\beta^{\R}(\alpha)$ on $(X \times Y^N, \mu \times \nu^N)$ called the \emph{coinduced equivalence relation}, whose construction we will briefly recall (see also \cite[Section 3]{IKT08}). 

Let $\{C_j\}_{j \in N} \subset [\R]$ with $C_0 = \id$ and such that for almost every $x \in X$, the sequence $\{C_j(x)\}_{j \in N}$ contains exactly one member of each $\R_0$-class contained in $[x]_\R$. These are called {\it choice functions} (see \cite[Lemmas 1.1 and 1.3]{FSZ89} for proof of their existence). 
For almost every $x \in X$, this gives us a way to number the $\R_0$-classes contained in $[x]_\R$. If $(x, x') \in \R$, then $x'$ will give rise to a new numbering of the $\R_0$-classes in $[x']_\R = [x]_\R$ and hence a permutation $\pi(x, x') \in S_N$ defined by
\begin{align}
n = \pi(x, x')(k)  \iff  [C_n(x)]_{\R_0} = [C_k(x')]_{\R_0}
\end{align}
which satisfies $\pi(x, x')\pi(x', x'') = \pi(x, x'')$ for almost every $(x, x'), (x', x'') \in \R$. Since $\beta$ is free, we can then define $\deltab_{(x, x')} \in (\Gamma_0)^N$ by 
\begin{align}
C_{\pi(x, x')(k)}(x) = \deltab_{(x, x')}(k) \cdot C_{k}(x') \textt{for} k \in N.
\end{align}
For $\y \in Y^N$, let $y_n \in Y$ denote the $n$th component of $\y$. Then we can then define the co-induced equivalence relation $\R_\alpha$ on $(X \times Y^N, \mu \times \nu^N)$ by 
\begin{align*}
(x, \y) {\R_\alpha} (x', \y') \iff \left[x {\R} x' \text{ and } y_{\pi(x, x')(k)} = \deltab_{(x, x')}(k) \cdot y'_k \text{ for all $k \in N$}\right].
\end{align*}

Proposition \ref{P: cind props} below gives some important properties that this constructions satisfies. For clarity in its proof, we first isolate the following basic fact as a lemma:

\begin{lemma}\label{L: basic}
Let $\H_1$ and $\H_2$ be Hilbert spaces, $\H = \H_1 \ten \H_2$, $\{\xi_n\} \subset \H_1$, $\{\eta_n\} \subset \H_2$ such that $\sup_{n} \|\xi_n\| < \infty$ and $\eta_n \to 0$ weakly in $\H_2$. Then $\xi_n \ten \eta_n \to 0$ weakly in $\H$. 
\end{lemma}
\begin{proof}
Note that $\sup_{n} \|\eta_n\| < \infty$ by the uniform boundedness principle, so $\{\xi_n \ten \eta_n\}$ is bounded and it is enough to check that 
$|\<\xi_n \ten \eta_n, \xi \ten \eta\>| \le  \|\xi\| \cdot \sup_k \|\xi_k\|\cdot|\<\eta_n, \eta\>|  \to 0$ as $n \to \infty$ for each $\xi \in \H_1$, $\eta \in \H_2$. 
\end{proof}

\begin{proposition}\label{P: cind props}
Let $\Gamma_0 \on{\beta} (X, \mu)$ be a free ergodic pmp action and $\R$ an ergodic pmp equivalence relation on $(X, \mu)$ such that $\R(\Gamma_0 \on{\beta} X) \leq \R$. Then for any pmp action $\Gamma_0 \on{\alpha} (Y, \nu)$ with $\nu$ nonatomic, the coinduced equivalence relation $\R_\alpha = {\rm CInd}_\beta^{\R}(\alpha)$ satisfies:
\begin{enumerate}
\item $\R_\alpha$ is an extension of $\R$. 
\item If $\alpha$ is weakly mixing, then $\R_\alpha$ is ergodic. 
\item If $\alpha$ is free, then $\R_\alpha$ is an expansion of $\R(\Gamma_0 \on{\alpha} Y)$.
\end{enumerate}
\end{proposition}

\begin{remark}
Assume that $\mathcal R$ is the orbit equivalence relation of some free pmp action $\Gamma\curvearrowright^{\sigma} (X,\mu)$ of a countable group $\Gamma$. Let $\Gamma\curvearrowright^{\tau} (X\times Y^N,\mu\times\nu^N)$ be the co-induced action of $\alpha$, modulo $(\beta,\sigma)$ (see \cite{Ep07} and also \cite[Section 3 (A)]{IKT08}, where this terminology is defined). Then $\mathcal R_{\alpha}$ is precisely the orbit equivalence relation of $\tau$.
In particular, if $\alpha$ is weakly mixing, then Proposition \ref{P: cind props} (2) implies that $\tau$ is ergodic.

This fact allows to simplify the proof of \cite[Lemma 2.6]{Ep07}. Indeed, in the context from \cite[Lemma 2.6]{Ep07}, it follows that the action $c$ of $\Gamma$ obtained by coinducing the weakly mixing action $a\times a_{\pi}$ of $\mathbb F_2$ modulo $(a_0,b_0)$ is ergodic, hence the use of the ergodic decomposition of $c$ is redundant.
\end{remark}

\begin{proof}{\rm (1)}.
Consider the measurable map $p: X \times Y^N \to X$ defined by $p(x, \y) = x$. Then $\mu = [\mu \times \nu^N] \circ p^{-1}$ and for $(x, \y) \in X \times Y^N$ we have $p([(x, \y)]_{\R_\alpha}) = [x]_\R$ injectively since $\pi_{(x, x)} = \id$, $\deltab_{(x, x)} = \id^N$. 

 {\rm (2)}.
Let $E \subset X \times Y^N$ be an $\R_\alpha$-invariant Borel subset and let $1_E$ denote the characteristic function viewed as an element of $L^2(X) \ten \bigotimes_{k \in N} L^2(Y) \cong L^2(X \times Y^N)$, where the tensor product is taken with respect to the reference vector $1 \in L^2(Y)$ in each component. Then $\sigma_{\theta}(1_E)=1_E$ for all $\theta\in [\mathcal R_{\alpha}]$, where we define $\sigma_{\theta}(\xi)=\xi\circ\theta^{-1}$ for $\xi\in L^2(X\times Y^N)$.
For $\theta\in[\mathcal R]$, let $\tilde\theta\in[\mathcal R_{\alpha}]$ be its {\it lift}, i.e. the unique element in $[\mathcal R_{\alpha}]$ such that $p\circ\tilde\theta=\theta\circ p$. 

Denote by $I \subset (\ZZ_{\ge 0})^N$ the subset consisting of $(i_j)_{j \in N}$ such that $i_j = 0$ for all but finitely many $j \in N$. 
Let $\{\eta_i\}_{i = 0}^\infty$ be an orthonormal basis of $L^2(Y)$ with $\eta_0 = 1$, and for $\ii = (i_j)_{j \in N} \in I$, let $\eta_\ii = \bigotimes_{j \in N} \eta_{i_j}$. 
Then expanding $1_E = \sum_{\ii \in I} \xi_\ii \ten \eta_\ii$ with $\xi_\ii \in L^2(X)$, we will show that $1_E \in L^2(X) \ten \CC$ by showing that $\xi_\ii = 0$ for any $\ii \in I$ which has $i_k \ne 0$ for some $k \in N$. This will finish the proof. Indeed, since $1_E\in L^2(X)\ten\CC$ is $\mathcal R_{\alpha}$-invariant, it follows that $1_E$ is $\mathcal R$-invariant. Since $\R$ is ergodic, this will then force $1_E \in \CC$, i.e., $\mu(E) \in \{0, 1\}$.

Fix such $\ii$ with $i_k \ne 0$. Since $\alpha$ is weakly mixing and $\eta_{i_k} \perp \CC$, there is a sequence $\{g_n\}_{n = 1}^\infty \subset \Gamma_0$ such that $\alpha_{g_n}(\eta_{i_k}) \to 0$ weakly in $L^2(Y)$ as $n \to \infty$. Let $\{C_j\}_{j = 0}^N \subset [\R]$ be the choice functions used to construct $\R_\alpha$, and for each $n \ge 1$, set $\theta_n = C_{k}^{-1}\circ g_n \circ C_k \in [\R]$. Then $C_k(x) = g_n \cdot C_k(\theta_n^{-1}x)$  
and so 
\begin{align*}
\pi_{(x, \theta_n^{-1} x)}(k) = k \quad \text{and} \quad
\deltab_{(x, \theta_n^{-1} x)}(k) = g_n \quad \text{for all} \quad n \ge 1.
\end{align*}

Hence $y'_k = g_n^{-1}y_k$ for $n \ge 1$ and $x, x' \in X$, $\y, \y' \in Y^N$ with
${\tilde{\theta}_n}^{-1} \cdot (x, \y) = (x', \y')$. 
Therefore defining
$\zeta_n \in L^2(X) \ten \bigotimes_{j \in N \setminus \{k\}} L^2(Y)$ 
by 
\begin{align*}
\zeta_n(x, y_0, \dots, \hat{y_k}, \dots) 
= \xi_\ii(\theta_n^{-1}x) \ten \bigotimes_{j \in N \setminus \{k\}} \eta_{i_j}(y'_j) \textt{where} (x', \y') = \tilde{\theta_n}^{-1} \cdot (x, \y)
\end{align*}
we have 
$\sigma_{\tilde\theta_n}(\xi_\ii \ten \eta_\ii) = \zeta_n \ten \alpha_{g_n}(\eta_{i_k})$ with $\|\zeta_n\| = \|\xi_\ii\|$. 
Then for any $n$,
\begin{align*}
\|\xi_\ii\|^2 
= \|\xi_\ii \ten \eta_\ii\|^2 
= \<1_E, \xi_\ii \ten \eta_\ii\> 
= \<\sigma_{\tilde\theta_n}^{-1}(1_E), \xi_\ii \ten \eta_\ii\>
= \<1_E, \sigma_{\tilde\theta_n}(\xi_\ii \ten \eta_\ii)\>
= \<1_E, \zeta_n \ten \alpha_{g_n}(\eta_{i_k})\>
\end{align*}
and $\<1_E, \zeta_n \ten \alpha_{g_n}(\eta_{i_k})\> \to 0$ as $n \to \infty$ by Lemma \ref{L: basic}, so we indeed have $\xi_\ii = 0$. 

{\rm (3)}.
Consider the surjection $p: X \times Y^N \to Y$ by $p(x, \y) = y_0$. Then $\nu = \mu \times \nu^N \circ p^{-1}$ and $p([(x, \y)]_{\R_\alpha}) \supset [p(x, \y)]_{\R(\Gamma_0 \on{\alpha} Y)}$. 
 
Take any $(x, \y) \in X \times Y^N$ and suppose that $(x', \y') \ne (x'', \y'')$ are members of $[(x, \y)]_{\R_\alpha}$ with $y_0' = y_0''$.
Then for $k = \pi(x, x')(0)$ and $m = \pi(x, x'')(0)$ we have 
\begin{align*}
y_{k} = \deltab_{(x, x')}(0)y_0' 
= \deltab_{(x, x')}(0)y_0'' 
= \deltab_{(x, x')}(0)\deltab_{(x, x'')}(0)^{-1}y_m.
\end{align*}
If $k = m$ and $g = \deltab_{(x, x')}(0)\deltab_{(x, x'')}(0)^{-1} = e$ (the identity of $\Gamma_0$), then 
\begin{align*}
x' = C_0(x') = \deltab_{(x, x')}(0)^{-1}C_k(x) = \deltab_{(x, x'')}(0)^{-1}C_k(x) = C_0(x'') = x''
\end{align*}
which would contradict $(x', \y') \ne (x'', \y'')$. On the other hand, if $k = m$ and $g \ne e$, then by the freeness of $\Gamma_0 \on{\alpha} (Y, \nu)$,
\begin{align*}
(\mu \times \nu^N)(\{(x, \y) \in X \times Y^N: y_k = gy_k\}) 
= \nu(\{y \in Y: y = gy\}) = 0.
\end{align*} 
Hence
\begin{align*}
&(\mu \times \nu^N)(\{(x, \y) \in X \times Y^N: p \text{ is not injective on } [(x, \y)]_{\R_\alpha}\}) \\
\le &\sum_{k \ne m \in N} \sum_{g \in \Gamma_0} (\mu \times \nu^N)(\{(x, \y) \in X \times Y^N: y_k = gy_m\}) \\
= &\sum_{k \ne m \in N} \sum_{g \in \Gamma_0} \int_{X \times Y^{N-1}} \nu(\{gy_m\})d(\mu \times \nu^{N-1})(x, (y_0, \dots, \hat y_k, \dots))
= 0
\end{align*}
since $\nu$ is non-atomic. 
\end{proof}
\end{Hoff_Commands}

\subsection{A separability argument}
\begin{Hoff_Commands}
Let $\lambda$ denote the Haar measure on $\mathbb{T}$.
Let $SL_2(\mathbb Z)\curvearrowright (\mathbb T^2,\lambda^2)$ be the pmp action given by matrix multiplication. Consider a fixed embedding of $\mathbb F_2$ as a finite index subgroup of $SL_2(\mathbb Z)$. Then the restricted action $\FF_2 \on{\alpha^0} (\mathbb{T}^2, \lambda^2)$ is free, weakly mixing, and {\it rigid}, in the sense of S. Popa \cite[Corollary 5.2]{Po01}.
The latter means that the inclusion of von Neumann algebras $L^{\infty}(\mathbb T^2)\subset L^{\infty}(\mathbb T^2)\rtimes\mathbb F_2$ has relative property (T), as defined in \cite[Definition 4.2]{Po01}.

If an equivalence relation $\mathcal R$ on $(X,\mu)$ is an expansion of $\R(\FF_2 \on{\alpha^0} \mathbb{T}^2)$, then there is a canonical way to define an extension $\mathbb F_2\curvearrowright^{\sigma} X$ of $\alpha^0$ whose orbit equivalence relation is contained in $\mathcal R$.
Specifically, if $p:X\rightarrow\mathbb T^2$ denotes the quotient map, then $\sigma$ is the unique  such action satisfying $p\circ\sigma(g)=\alpha^0(g)\circ p$, for every $g\in\mathbb F_2$.

\begin{lemma} \label{L: sep}
Let $\{\R_i\}_{i \in I}$ on $\{(X_i, \mu_i)\}_{i \in I}$ be an uncountable collection of stably von Neumann equivalent ergodic pmp countable equivalence relations, each an expansion of $\R(\FF_2 \on{\alpha^0} \mathbb{T}^2)$. For each $i \in I$, let $\FF_2 \on{\sigma^i} X_i$ denote the canonical extension of $\alpha^0$ with $\mathcal R(\FF_2 \on{\sigma^i} X_i)\leq\mathcal R_i$. 

Then there exists an uncountable set $J \subset I$ such that for any $i, j \in J$ there is a $\sigma^i$-invariant (resp. $\sigma^j$-invariant)  non-null Borel set  $E_i \subset X_i$ (resp. $E_j \subset X_j$) with the restricted actions $\sigma^i|_{E_i}$ and $\sigma^j|_{E_j}$ conjugate. 
\end{lemma}

Lemma \ref{L: sep} is an analogue of \cite[Theorems 1.3 and 4.7]{Io06} for equivalence relations.
Its proof combines relative property (T) with a separability argument. Property (T) was first employed in the context of von Neumann algebras by A. Connes in \cite{Co80}. 
The original idea of combining property (T) and its relative version with a separability argument is due to S. Popa \cite{Po86}. It has since proven greatly influential and has been successfully used in various contexts, including in the work of D. Gaboriau and S. Popa in \cite{Po01,GP03}.

\begin{proof}
Since $\{\R_i\}_{i\in I}$ are stably von Neumann equivalent, after replacing $I$ with an uncountable subset, we may find a separable ${\rm II}_1$ factor $M$ and non-zero projections $p_i \in L(\R_i)$ such that $M \cong p_iL(\R_i)p_i$, for all $i \in I$. 
We denote by $\tau$ and $\|.\|_2$ the trace and $2$-norm on $M$, and by $\tau_i$ the trace on $L(\mathcal R_i)$.
For each $i \in I$, let $B_i = L^\infty(X_i)$ and $N_i = L^\infty(X_i) \rtimes_{\sigma^i} \FF_2$, regarded  as subalgebras of $L(\R_i)$.
Let $A = L^\infty(\mathbb{T}^2)$ and $Q = L^\infty(\mathbb{T}^2) \rtimes_{\alpha^0} \FF_2$.
We have copies $A_i \cong A$, $Q_i \cong Q$ with $A_i \subset B_i$, $Q_i \subset N_i$, and by Lemma \ref{L: exp}, $A_i' \cap L(\mathcal R_i) = B_i$ for each $i\in I$.

Since $I$ is uncountable, we can find $t\in (0,1]$ such that  $I_\epsilon = \{i \in I| 1-\epsilon^2 \le t/\tau_i(p_i)\le 1\}$ is uncountable, for all $\epsilon > 0$. 
As $A$ is diffuse, there is a projection $q \in A$ such that $\tau(q)=t$. Since $Aq \subset qQq$ has relative property (T), there is a finite set $F \subset (qQq)_1$ and $\delta > 0$ such that for any $qQq$-$qQq$ bimodule $\H$ with nonzero $\xi_0 \in \H$ satisfying $\|x\xi_0 - \xi_0 x\| < \delta \|\xi_0\|$ for all $x \in F$, there is nonzero $\xi \in \H$ with $a\xi = \xi a$ for all $a \in Aq$.
Let $\epsilon > 0$ small enough that $\frac{3\epsilon}{1-2\epsilon} < \delta$ and set $I_1 = I_\epsilon$. 

For $x \in Q$, we let $x_i \in Q_i$ denote the image in $Q_i$.
Each $L(\R_i)$ is a factor, so by conjugating by a unitary in each, we may assume that $q_i \le p_i$, for all $i \in I_1$. Then identifying $p_iL(\R_i)p_i$ with $M$, we have $q_i \in M$ and $\tau(q_i) \ge 1-\epsilon^2$ so that $\|1_M - q_i\|_2 \le \epsilon$, for each $i \in I_1$. 

Then for any $i, j \in I_1$, endow $q_iL^2(M)q_j$ with a $qQq$-$qQq$ bimodule structure given by defining $x \cdot \xi \cdot y = x_i\xi y_j$, for all $x, y \in qQq$ and $\xi \in q_iL^2(M)q_j$. 
 Let $\xi_{i, j} = q_iq_j \in q_iL^2(M)q_j$ and note that $\|\xi_{i, j} - 1_M\|_2 \le \|1_M - q_i\|_2 + \|1_M - q_j\|_2 \le 2\epsilon$ and hence $\|\xi_{i, j}\|_2 \ge 1 - 2\epsilon$.
 
Since $M$ is $\|\cdot\|_2$-separable, there is an uncountable set $J \subset I_1$ such that $\|x_i - x_j\|_2 < \epsilon$ for all $i, j \in J$ and $x \in F$. Fix any $i, j \in J$. Then for any $x \in F$,
\begin{align*}
\|x_i\xi_{i, j} - \xi_{i, j}x_j\|_2 
\le \|x_i - x_j\|_2 + \|1_M - q_i\|_2 + \|1_M - q_j\|_2 
\le 3\epsilon < \delta(1-2\epsilon) \le \delta \|\xi_{i, j}\|_2,
\end{align*}

and so by relative property (T) there is nonzero $\xi \in q_iL^2(M)q_j$ with $a_i\xi = \xi a_j$ for all $a \in Aq$. Then the polar decomposition $\xi = v|\xi|$ has $v \in M$ with $a_iv = v a_j$ for all $a \in Aq$. Set $e_i = vv^*$ and $e_j = v^*v$. 

For any $b \in B_jq_j$ and $a \in Aq$ we have $a_ivbv^* = va_jbv^* = vba_jv^* = vbv^*a_i$, so $vB_jv^* \subset (A_iq_i)' \cap q_iMq_i = B_iq_i$ and similarly $v^*B_iv \subset B_jq_j$, and in particular, $e_i \in B_i$, $e_j \in B_j$.  
We thus define a trace preserving $*$-isomorphism $\Psi: B_je_j \to B_ie_i$ by $b \mapsto vbv^*$. 

Then for positive measure sets $F_i \subset X_i$ and $F_j \subset X_j$ with $e_i = 1_{F_i}$ and $e_j = 1_{F_j}$, there is a measure space isomorphism $\Theta: (F_i, \mu_i) \to (F_j, \mu_j)$ such that 
$\Psi(b) = b \circ \Theta$ for all $b \in B_j$. 
Let $E_i = \bigcup_{g \in \FF_2} \sigma^i_g(F_i)$ and $E_j = \bigcup_{g \in \FF_2} \sigma^j_g(F_j)$. 
Then $E_i$ is $\sigma^i$-invariant, $E_j$ is $\sigma^j$-invariant and 
we will show that $\Theta$ can be extended to a measure space isomorphism $\Theta: (E_i, \mu_i) \to (E_j, \mu_j)$ by the formula 
\begin{align}\label{E: extgoal}
\Theta(x) = [\sigma_g^j \circ \Theta \circ \sigma_{g^{-1}}^i](x) \quad \text{for} \quad x \in \sigma^i_g(F_i), \;g \in \FF_2
\end{align}
which will then satisfy $[\sigma_g^j \circ \Theta](x) = [\Theta \circ \sigma_{g}^i](x)$ for $x \in E_i$, $g \in \FF_2$, showing that $\sigma^i|_{E_i}$ and $\sigma^j|_{E_j}$ are conjugate.
Toward showing that \eqref{E: extgoal} is well defined, for $g \in \FF_2$ let $u_{g_i} \in Q_i$ and $u_{g_j} \in Q_j$ denote respectively the canonical unitaries implementing $\sigma^i$ and $\sigma^j$. Viewing $v^*u_{g_i}^*v \in e_jL(\R_j)e_j \subset L(\R_j)$, for $a \in A$ we have 
\begin{align*}
a_ju_{g_j}v^*u_{g_i}^*v 
= u_{g_j}\sigma^j_{g^{-1}}(a_j)v^*u_{g_i}^*v 
= u_{g_j}v^*\sigma^i_{g^{-1}}(a_i)u_{g_i}^*v
= u_{g_j}v^*u_{g_i}^*a_iv
= u_{g_j}v^*u_{g_i}^*va_j
\end{align*}
so that $u_{g_j}v^*u_{g_i}^*v \in A_j' \cap L(\R_j) = B_j$. Therefore for any $b \in B_j$, we have
\begin{align*}
u_{g_i}vu_{g_j}^*b(u_{g_j}v^*u_{g_i}^*v)v^*
= u_{g_i}vu_{g_j}^*(u_{g_j}v^*u_{g_i}^*v)bv^*
= (u_{g_i}e_iu_{g_i}^*)vbv^*
\end{align*}
and hence 
\begin{align*}
\sigma_g^i(v\sigma_{g^{-1}}^j(b)v^*)e_i = vbv^*\sigma_g^i(e_i) \quad \text{for all} \quad b \in B_j, g \in \FF_2,
\end{align*}
which when applied to $h^{-1}g \in \FF_2$ for $g, h \in \FF_2$ gives
\begin{align*}
\sigma_g^i(v\sigma_{g^{-1}}^j(b)v^*)\sigma^i_h(e_i) = \sigma_h^i(v\sigma_{h^{-1}}^j(b)v^*)\sigma_g^i(e_i) \quad \text{for all} \quad b \in \sigma_{h}^j(B_j) = B_j,
\end{align*}
which translates to
\begin{align*}
[\sigma_g^j \circ \Theta \circ \sigma_{g^{-1}}^i](x) = [\sigma_h^j \circ \Theta \circ \sigma_{h^{-1}}^i](x) \quad \text{for all} \quad x \in \sigma^i_g(F_i) \cap \sigma^i_h(F_i)
\end{align*}
showing that \eqref{E: extgoal} is well defined.
\end{proof}
\end{Hoff_Commands}

\subsection{Proof of Theorem \ref{T: ext}}
\begin{Hoff_Commands}

 Let $\mathcal R$ be a non-amenable ergodic countable pmp equivalence relation on a probability space $(X,\mu)$. 
Our goal is to show that $\mathcal R$ has uncountably many ergodic extensions which are pairwise not stably von Neumann equivalent.
Below, for a pmp action $\mathbb F_2\curvearrowright^{\alpha} (Y,\nu)$, we denote by $\pi_{\alpha}^0$ and $\pi_{\alpha}$ the Koopman representations of $\mathbb F_2$ on $L^2(Y)\ominus\mathbb C1$ and $L^2(Y)$, respectively.

Let $\tilde \R$ on $(\tilde X, \tilde \mu)$ denote the Bernoulli extension of $\R$ with base space $([0,1],\lambda)$. By Theorem \ref{main}, there is a free ergodic pmp action $\FF_2 \on{\beta} \tilde X$ such that $\R_0: = \R(\FF_2 \on{\beta} \tilde X) \leq \tilde \R$. 

By \cite{Sz88} there is an uncountable family $\{\pi_i: \FF_2 \to \U(H_i)\}_{i \in I}$ of non-equivalent irreducible representations who are mixing, i.e. $\<\pi_i(g)\xi, \eta\> \to 0$ as $g \to \infty$ for any $\xi, \eta \in \H_i$. 
By considering the Gaussian action corresponding to the realification of $\pi_i$ (as in \cite{Ke10}, for example), we obtain an uncountable family of actions $\{\FF_2 \on{\alpha^i} (Y_i, \mu_i)\}_{i \in I}$ such that $\pi_i \subset \pi_{\alpha^i}^0$, for each $i \in I$. 

For each $i \in I$, note that $\alpha^i \times \alpha^0$ is weakly mixing since $\alpha^i$ is mixing and $\alpha^0$ is weakly mixing. By Proposition \ref{P: cind props}, $\tilde \R_i = {\rm CInd}_\beta^{\tilde \R}(\alpha^i \times \alpha^0)$ on $(\tilde Y_i, \tilde \nu_i)$ is an ergodic extension of $\tilde \R$ and hence of $\R$. Thus, we are done, unless uncountably many of the $\tilde \R_i$ are stably von Neumann equivalent. 
Therefore, assume toward a contradiction that there is an uncountable subset $I_0 \subset I$ such that the $\{\tilde \R_i\}_{i \in I_0}$ are stably von Neumann equivalent. 

By Proposition \ref{P: cind props}, each $\tilde\R_i$ is an expansion of $\mathcal R(\mathbb F_2\curvearrowright^{\alpha^i \times \alpha^0}Y_i\times\mathbb T^2)$ and hence of $\mathcal R(\mathbb F_2\curvearrowright^{\alpha^0}\mathbb T^2)$. Let $\FF_2 \on{\sigma^i} \tilde Y_i$ denote the canonical extension of $\alpha^0$.
Then by Lemma \ref{L: sep}, there is an uncountable subset $J \subset I_0$ such that for each $i, j \in J$ there is a $\sigma^i$-invariant (resp. $\sigma^j$-invariant) positive measure set $E_i \subset \tilde Y_i$ (resp. $E_j \subset \tilde Y_j$) with the restricted actions $\sigma^i|_{E_i}$ and $\sigma^j|_{E_j}$ conjugate. 

Since $\sigma^i$ is an extension of the ergodic action $\alpha^i \times \alpha^0$ of $\mathbb F_2$, $\sigma^i|_{E_i}$ is also an extension thereof. Hence, for all $i, j \in J$,
\begin{align*}
\pi_i \subset \pi^0_{\alpha^i} \subset \pi^0_{\alpha^i \times \alpha^0} \subset \pi^0_{\sigma^i|_{E_i}} \cong \pi^0_{\sigma^j|_{E_j}}\subset\pi_{\sigma^j}
\end{align*}
so that $\pi_{\sigma^j}$ has uncountably many nonequivalent irreducible sub-representations, contradicting the separability of $L^2(E_j)$. 
\hfill$\square$
\end{Hoff_Commands}


\section{Actions of locally compact groups} In this section we  prove Theorem \ref{lc} and explain how Theorem \ref{main} implies \cite[Theorem B]{GM15}. We begin by recalling the notion of cross section of actions of lcsc groups (see \cite[Definition 4.1]{KPV13}).

\begin{definition} Let $G$ be a lcsc group and $G\curvearrowright (X,\mu)$ a free nonsingular action on a standard probability space $(X,\mu)$. 
 A Borel set $Y\subset X$ is called a {\it cross section} of  $G\curvearrowright (X,\mu)$ if there exists a neighborhood $U$ of the identity in $G$ such that the map $U\times Y\rightarrow X$ given $(g,y)\mapsto gy$ is injective, and $\mu(X\setminus G\cdot Y)=0$. 
A cross section $Y\subset X$ is called {\it co-compact} if there is a compact set $K\subset G$ such that $K\cdot Y$ is a $G$-invariant Borel set and 
$\mu(X\setminus K\cdot Y)=0$. 
\end{definition}

\begin{remark} Assume that $G$ is a lcsc unimodular group. Let  $G\curvearrowright (X,\mu)$ be a free pmp action and $Y\subset X$ a cross section.
Then $\mathcal R=\{(y,y')\in Y\times Y|Gy=Gy'\}$ defines a countable Borel equivalence relation, called the {\it cross section equivalence relation}. 
Moreover, if $\lambda$ is a fixed Haar measure of $G$, then there exist a unique  $\mathcal R$-invariant probability measure $\nu$ on $Y$  and constant $c\in (0,+\infty)$ such that for every neighborhood $U$ of the identity in $G$ such that the map $\zeta:U\times Y\rightarrow X$ given by $\zeta(g,y)=gy$ is injective, we have $\zeta_*({\lambda}_{|U}\times\nu)=c\;\mu_{|U\cdot Y}$ (see \cite[Proposition 4.3]{KPV13}). Hereafter, we refer to $\nu$ as the {\it canonical} $\mathcal R$-invariant probability measure on $Y$.
\end{remark}

We continue with an elementary result which gives a construction of actions of locally compact groups with prescribed cross section equivalence relations.

\begin{proposition}\label{cross}
Let $G$ be a lcsc unimodular group and $G\curvearrowright (X,\mu)$ a free pmp action. Let $Y\subset X$ be a co-compact cross section of $G\curvearrowright (X,\mu)$, $\mathcal R$ be the cross section equivalence relation, and $\nu$ be the canonical $\mathcal R$-invariant probability measure on $Y$. Let $\bar{\mathcal R}$ be a countable pmp extension of $\mathcal R$ on a standard probability space $(\bar{Y},\bar{\nu})$.

Then there exist a free pmp action $G\curvearrowright (\tilde X,\tilde\mu)$,  and a co-compact cross section $\tilde Y\subset\tilde X$  such that the following holds. Denote by $\tilde{\mathcal R}$ the cross section equivalence relation on $\tilde Y$, and endow $\tilde Y$ with the canonical $\tilde{\mathcal R}$-invariant probability measure $\tilde\nu$. Then $\tilde{\mathcal R}$ is isomorphic to $\bar{\mathcal R}$.

\end{proposition}

{\it Proof.}  
Let $X_1\subset X$  be the set of points with trivial stabilizer. Then $X_1$ is a $G$-invariant Borel set (see e.g. \cite[Lemma 10]{MRV11}). Moreover, the freeness assumption implies that $X_1\subset X$ is co-null. Let $K\subset G$ be a compact set such that $X_2=K\cdot Y$ is a co-null $G$-invariant Borel subset of $X$.

 Put $X_0:=X_1\cap X_2$ and $Y_0:=Y\cap X_0$. Then  $X_0\subset X$ is a co-null $G$-invariant Borel subset,
$Y_0\subset Y$ is an $\mathcal R$-invariant Borel subset, and $K\cdot Y_0=X_0$. Let $U$ be a neighborhood of the identity in $G$ such that the map $U\times Y\rightarrow X$ given $(g,y)\mapsto gy$ is injective. Since $U\cdot (Y\setminus Y_0)$ is contained in $X\setminus X_0$, it is a null set. Let $\lambda$ be a  Haar measure of $G$. Since $\lambda(U)\nu(Y\setminus Y_0)=c\;\mu(U\cdot(Y\setminus Y_0))=0$, for some  $c>0$, and $\lambda(U)>0$, we get that $Y_0$ is co-null in $Y$. 

Altogether, we have that  $G\curvearrowright (X_0,\mu_{|X_0})$ is a pmp action such that every point has trivial stabilizer, $Y_0\subset X_0$ is a co-compact cross section with $K\cdot Y_0=X_0$,  $\mathcal R\resto Y_0$ is the associated cross section equivalence relation, and $\nu_{|Y_0}$ is the canonical $\mathcal R\resto Y_0$-invariant probability measure on $Y_0$. Moreover, since $Y_0\subset Y$ is co-null, $\mathcal R\resto Y_0$ is isomorphic to $\mathcal R$.
Thus, after replacing $X$, $Y$ with $X_0$, $Y_0$,
 we may assume that the stabilizer of every point in $X$ is trivial, and $K\cdot Y=X$, for a compact set $K\subset G$.  
 
Let $U$ be a neighborhood of the identity in $G$ such that the map $U\times Y\rightarrow X$ given $(g,y)\mapsto gy$ is injective. 
Define $\pi:U\cdot Y\rightarrow Y$ by letting $\pi(gy)=y$. Since $K$ is compact, we can find $g_1,...,g_n\in G$ such that $K\subset\cup_{i=1}^ng_iU$. Hence $X=K\cdot Y\subset \cup_{i=1}^ng_iU\cdot Y$. It follows that we can extend $\pi$ to a Borel map $\pi:X\rightarrow Y$ in such a way that $\pi(x)\in Gx$, for every $x\in X$.

Let $p:\bar{Y}\rightarrow Y$ be the quotient map. After replacing $\bar{Y}$ with a co-null $\bar{\mathcal R}$-invariant Borel subset, we may assume that $p|_{[\bar{y}]_{\bar{\mathcal R}}}$ is injective and $p([\bar{y}]_{\bar{\mathcal R}})=[p(\bar{y})]_{\mathcal R}$, for all $\bar{y}\in\bar{Y}$.

Let $\tilde X=X\times_{Y}\bar{Y}$ be the ``fibered product" Borel space given by $\tilde X=\{(x,\bar{y})\in X\times\bar{Y}|\pi(x)=p(\bar{y})\}$. We define a free Borel action $G\curvearrowright\tilde X$ as follows. Let $g\in G$ and $(x,\bar{y})\in\tilde X$. Since $\pi(gx)\in Gx\cap Y$, we get that $\pi(gx)\in [\pi(x)]_{\mathcal R}=[p(\bar{y})]_{\mathcal R}$. Thus, there is a unique $\hat{y}\in [\bar{y}]_{\bar{\mathcal R}}$ such that $p(\hat{y})=\pi(gx)$.
Finally, we let $g (x,\bar{y})=(gx,\hat{y})$. It is easy to check that this indeed defines a Borel action of $G$.

Next, let $\tilde Y=\{(p(\bar{y}),\bar{y})|\bar{y}\in\bar{Y}\}$. Then $\tilde Y$ is a Borel subset of $\tilde X$.
Let $(g_1,\bar{y}_1),(g_2,\bar{y}_2)\in U\times\bar{Y}$ such that $g_1 (p(\bar{y}_1),\bar{y}_1)=g_2 (p(\bar{y}_2),\bar{y}_2)$. Then $g_1p(\bar{y}_1)=g_2p(\bar{y}_2)$ and since $p(\bar{y}_1),p(\bar{y}_2)\in Y$, we deduce that $g_1=g_2$, which implies that $y_1=y_2$. Thus, the map $U\times\tilde Y\rightarrow \tilde X$ given by $(g,y)\mapsto g y$ is injective. Let $(x,\bar{y})\in\tilde X$. Since $K\cdot Y=X$, we can find $g\in K$ such that $g^{-1}x\in Y$.
 Let $\hat{y}\in [\bar{y}]_{\bar{\mathcal R}}$ such that $g^{-1}(x,\bar{y})=(g^{-1}x,\hat{y})$. Since $p(\hat{y})=\pi(g^{-1}x)=g^{-1}x$, we deduce that $g^{-1}(x,\bar{y})\in\tilde Y$. Thus, $K\cdot\tilde Y=\tilde X$. This  proves that $\tilde Y$ is  a co-compact cross section for the Borel action $G\curvearrowright\tilde X$. 
In particular, the first paragraph of the proof implies that there exists a Borel map $\tilde\pi:\tilde X\rightarrow\tilde Y$ such that $\tilde\pi(y)=y$, for every $y\in\tilde Y$, and $\tilde\pi(x)\in Gx$, for every $x\in\tilde X$.

Further, consider the cross section equivalence relation $\tilde{\mathcal R}=\{(y,y')\in\tilde Y\times\tilde Y|Gy=Gy'\}$.
Let $\theta:\bar{Y}\rightarrow\tilde Y$ be the Borel isomorphism given by $\theta(\bar{y})=(p(\bar{y}),\bar{y})$. It is easy to see that $(\theta\times\theta)(\bar{\mathcal R})=\tilde{\mathcal R}$. We endow $\tilde Y$ with the probability measure $\tilde\nu=\theta_*\bar{\nu}$. Since $\bar{\nu}$ is $\bar{\mathcal R}$-invariant, $\tilde{\nu}$ is $\tilde{\mathcal R}$-invariant.

By \cite[Section 4.2]{Sl15}, the $\tilde{\mathcal R}$-invariant probability measure $\tilde{\nu}$ on the co-compact cross section $\tilde Y$ can be ``lifted" to a $G$-invariant finite measure $\tilde{\mu}$ on $\tilde X$.  Specifically, for a Borel set $A\subset\tilde X$, we have

$$\tilde\mu(A)=(\lambda\times\tilde\nu)(\{(g,\tilde y)\in G\times\tilde Y|\tilde\pi(g\tilde y)=\tilde y\;\;\text{and}\;\;g\tilde y\in A\}).$$

Since the map $U\times\tilde Y\rightarrow U\cdot\tilde Y$ given by $(g,y)\mapsto g y$ is a bijection, it follows that under this identification we have that $\tilde\mu_{|U\cdot\tilde Y}=\lambda_{|U}\times\tilde\nu$. By using \cite[Proposition 4.3]{KPV13} we conclude that $\tilde\nu$ is the canonical $\tilde{\mathcal R}$-invariant probability measure on the cross section $\tilde Y$ for the free pmp action $G\curvearrowright (\tilde X,\frac{1}{\tilde\mu(X)}\tilde\mu)$. This concludes the proof of the proposition.
\hfill$\square$

\subsection{Proof of Theorem \ref{lc}}
Let $G\curvearrowright (X,\mu)$ be a free ergodic pmp action (see \cite[Remark 1.1]{KPV13} for a proof of existence). By \cite[Theorem 4.2]{KPV13} we can find a co-compact cross section $Y$ of $G\curvearrowright (X,\mu)$. Denote by $\mathcal R$ the associated cross section equivalence relation, and endow $Y$ with the canonical $\mathcal R$-invariant probability measure $\nu$. Since $G$ is non-amenable and $G\curvearrowright (X,\mu)$ is ergodic, \cite[Proposition 4.3]{KPV13} gives that $\mathcal R$ is non-amenable and ergodic.

By Theorem \ref{T: ext}, we can find an uncountable family $\{\tilde{\mathcal R}_i\}_{i\in I}$ of countable ergodic pmp extensions of $\mathcal R$ which are pairwise not stably von Neumann equivalent. 
By Proposition \ref{cross}, for every $i\in I$ we can find a free ergodic pmp action $G\curvearrowright (\tilde X_i,\tilde{\mu}_i)$ and a co-compact cross section $\tilde Y_i\subset\tilde X_i$ such that the associated cross section equivalence relation is isomorphic to $\tilde{\mathcal R}_i$.

We claim that the actions $G\curvearrowright (\tilde X_i,\tilde\mu_i)$, $i\in I$, are pairwise not von Neumann equivalent. Indeed, assume that $L^{\infty}({\tilde X}_i)\rtimes G\cong L^{\infty}(\tilde X_j)\rtimes G$, for some $i\not=j$. On the other hand, $L^{\infty}(\tilde X_i)\rtimes G$ and $L^{\infty}(\tilde X_j)\rtimes G$ are  amplifications of the II$_1$ factors $L(\tilde{\mathcal R}_i)$ and $L(\tilde{\mathcal R}_j)$ by \cite[Lemma 4.5]{KPV13}. It follows that we can find non-zero projections $p_i\in L(\tilde{\mathcal R}_i)$ and $p_j\in L(\tilde{\mathcal R}_j)$ such that $p_iL(\tilde{\mathcal R}_i)p_i\cong p_jL(\tilde{\mathcal R}_j)p_j$. This contradicts the fact that $\tilde{\mathcal R}_i$ and $\tilde{\mathcal R}_j$ are not stably von Neumann equivalent.
\hfill$\square$

\subsection{Deducing  \cite[Theorem B]{GM15} from Theorem \ref{main}}\label{GM} Let $G$ be a non-amenable lcsc group. Let $\lambda$ be a Haar measure of $G$. 
To show the existence of a tychomorphism from $\mathbb F_2$ to $G$, in the sense of \cite[Definition 14]{GM15}, we first reduce to the case when $G$ is unimodular. 

Denote by $G_0$ the kernel of the modular homomorphism of $G$. Then $G_0$ is non-amenable. Moreover, $G_0$ is unimodular. Indeed, since $G_0<G$ is a closed normal subgroup,  $G/G_0$ is a locally compact group, thus it admits a $G$-invariant Borel measure. \cite[Corollary B.1.7.]{BdHV08} now implies that $G_0$ is unimodular.
Thus, by \cite[Proposition 18]{GM15}, we may assume that $G$ is unimodular. Since the conclusion follows from the Gaboriau-Lyons theorem in the discrete case, we may additionally assume that $G$ is not discrete.

 Let $G\curvearrowright (X,\mu)$ be a free ergodic pmp action, $Y$ a co-compact cross section, $\mathcal R$ the cross section equivalence relation, and $\nu$ the canonical $\mathcal R$-invariant probability measure on $Y$. Since $G$ is non-amenable and $G\curvearrowright (X,\mu)$ is ergodic, $\mathcal R$ is non-amenable and ergodic.  

 By  Theorem \ref{main} there exist a countable ergodic pmp extension $\tilde{\mathcal R}$ of $\mathcal R$ on a probability space $(\tilde Y,\tilde\nu)$ and a free ergodic pmp action $\mathbb F_2\curvearrowright (\tilde Y,\tilde\nu)$ such that $\mathbb F_2y\subset [y]_{\tilde{\mathcal R}}$, for all $y\in\tilde Y$. By Proposition \ref{cross}, we can be realize $\tilde Y$ as a co-compact cross section of some free ergodic pmp action $G\curvearrowright (\tilde X,\tilde\mu)$, such that $\tilde{\mathcal R}$ is precisely the associated cross section equivalence relation. Moreover,  the proof of Proposition \ref{cross} gives that any point in $\tilde X$ has trivial stabilizer and $K\cdot\tilde Y=\tilde X$, for $K\subset G$ compact.

Let $U\subset G$ be a neighborhood of the identity such that the map $\zeta:U\times\tilde Y\rightarrow\tilde X$ given by $\zeta(h,y)=hy$ is injective. Define $$\tilde X_0:=U\cdot\tilde Y\;\;\;\text{and}\;\;\;\mathcal D:=\{(x, x')\in\tilde X\times\tilde X_0)| Gx=Gx'\}.$$ Consider the obvious action of $G$ on $\mathcal D$ on the first coordinate.
As in the end of \cite[Section 5]{GM15}, we endow $\mathcal D$ with a $G$-invariant measure $m$ by pushing forward $\lambda\times\mu_{|\tilde X_0}$ through the identification $G\times\tilde X_0\rightarrow\mathcal D$ given by $(g,x)\mapsto (gx,x)$. Then $(\mathcal D,m)$ is a finite amplification of the $G$-space $(G,\lambda)$, in the sense of \cite[Definition 11]{GM15}.

Next, we define an $m$-preserving action $\mathbb F_2\curvearrowright\mathcal D$, as follows. Fix $\theta\in\mathbb F_2$. If $(x,x')\in\mathcal D$, then $x'\in\tilde X_0$, hence
 we can write $x'=hy$, for some $h\in U$ and $y\in\tilde Y$. We define $\tilde\theta(x,x')=(x,h\theta(y))$. Let $\alpha:\tilde Y_0\rightarrow G$ be given by $\theta(y)=\alpha(y)y$, for every $y\in\tilde Y_0$. Then in the above identification $G\times\tilde X_0\equiv\mathcal D$, $\tilde\theta$ corresponds to the Borel automorphism of $G\times\tilde X_0$ given by $(g,hy)\mapsto (gh\alpha(y)^{-1}h^{-1},h\theta(y)),$ for all $h\in U, y\in\tilde Y_0$. Since $G$ is unimodular, $\zeta_*(\lambda_{|U}\times\tilde\nu)=c\;\tilde\mu_{|X_0}$, for some $c>0$, and $\theta$ preserves $\nu$,  it follows that $\tilde\theta$ preserves $m$.

We claim that $\mathbb F_2\curvearrowright \mathcal D$ admits a non-null  measurable fundamental domain. Since the  actions of $\mathbb F_2$ and $G$ on $\mathcal D$ commute,
it will follow that $\mathcal D$ gives rise to a tychomorphism from $\mathbb F_2$ to $G$.
To prove the claim, since $K$ is compact, let $g_1,...,g_n\in G$ such that $K\subset\cup_{j=1}^ng_jU$. Then $\tilde X=K\cdot\tilde Y\subset\cup_{j=1}^ng_j\tilde X_0$ and therefore  $\mathcal D=\cup_{j=1}^ng_j\mathcal D_0$, where $\mathcal D_0:=\{(x,x')\in\tilde X_0\times\tilde X_0|Gx=Gx'\}$.

Since $\mathcal D_0$ is $\mathbb F_2$-invariant,  in order to prove the claim, it suffices to show that the action $\mathbb F_2\curvearrowright\mathcal D_0$ admits a non-null  measurable fundamental domain.
To see this, using that the action $\mathbb F_2\curvearrowright (\tilde Y,\tilde\nu)$ is ergodic, we choose a sequence $\{C_i\}_{i\geq 1}\subset [\tilde{\mathcal R}]$ such that $[y]_{\tilde{\mathcal R}}$ is the disjoint union of $\mathbb F_2C_i(y)$, $i\geq 1$, for almost every $y\in\tilde Y$ (see \cite[Remark 2.1]{IKT08}). Since $\mathcal D_0=\{(hy,h'y')|h,h'\in U, (y,y')\in\mathcal R\}$, one
 checks that $\mathcal F:=\{(hy,h'C_i(y))|h,h'\in U,y\in\tilde Y,i\geq 1\}$ is a non-null measurable fundamental domain for the action $\mathbb F_2\curvearrowright\mathcal D_0$.
\hfill$\square$

\end{document}